\documentclass[11pt]{amsart}
\usepackage{amsthm}

\theoremstyle{plain}
\newtheorem{thm}{Theorem}[section]
\newtheorem{theorem}[thm]{Theorem}

\newtheorem{lemma}[thm]{Lemma}

\newtheorem{proposition}[thm]{Proposition}
\theoremstyle{definition}

\newtheorem{remarks}[thm]{Remarks}

\newtheorem{definition}[thm]{Definition}

\numberwithin{equation}{section}

 \title[Strong submeasures and applications to non-compact dynamical systems]{Strong submeasures and applications to non-compact dynamical systems}
 \author{Tuyen Trung Truong}
   \address{Department of Mathematics, University of Oslo, Blindern 0851 Oslo, Norway}
  \email{tuyentt@math.uio.no}
    \date{\today}

\begin{document}
\maketitle

\begin{abstract}
A strong submeasure on a compact metric space X is a sub-linear and bounded operator on the space of continuous functions on X. A strong submeasure is positive if it is non-decreasing. By Hahn-Banach theorem, a positive strong submeasure is the supremum of a non-empty collection of measures whose masses are uniformly bounded from above. 

There are many natural examples of continuous maps of the forms $f:U\rightarrow X$, where $X$ is a compact metric space and $U\subset X$ is an open dense subsets, where $f$ cannot extend to a reasonable function on $X$. We can mention cases such as: transcendental maps of $\mathbb{C}$, meromorphic maps on compact complex varieties, or continuous self-maps $f:U\rightarrow U$ of a dense open subset $U\subset X$ where $X$ is a compact metric space.       

To the maps mentioned in the previous paragraph, the use of measures is not sufficient to establish the basic properties of ergodic theory, such as the existence of invariant measures or a reasonable definition of measure theoretic entropy and topological entropy. In this paper, we show that strong submeasures can be used to completely resolve the issue and establish these basic properties. The same idea can also be applied to improve on \'etale topological entropy, previously defined by the author. In another paper we apply strong submeasures to intersection of positive closed $(1,1)$ currents on compact K\"ahler manifolds.
\end{abstract}

\section{Introduction} In dynamical systems and ergodic theory, measures play a crucial role. At least since H. Poincar\'e, the first fundamental step for studying the dynamics of a continuous map $f:X\rightarrow X$ of a compact metric space $X$ is to construct invariant probability measures, that is those measures $\mu$ for which $f_*(\mu )=\mu$, and in particular those with measure entropy equal to the topological entropy. A way to construct invariant measures is to start from a positive measure $\mu _0$, and then to consider any cluster points of the Cesaro's average $\frac{1}{n}\sum _{j=0}^n(f_*)^j(\mu _0)$. To this end, a crucial property is that we can pushforward a probability measure by a continuous map and obtain another probability measure, and that this pushforward is linear on the space of measures. There is also the fundamental result \cite{goodwyn, goodman}, called the Variational Principle, which relates measure entropies of invariant measures of a compact metric space and the topological entropy of the map. When we work with compact complex varieties, usually it is very difficult to construct dynamically interesting holomorphic maps $f:X\rightarrow X$, and one must be willing to deal with dominant meromorphic maps $f:X\dashrightarrow X$ in order to go forward. We can still define a notion of entropy for meromorphic maps, and when $X$ is K\"ahler \cite{dinh-sibony3} relates this topological notion with geometrical/cohomological information (called dynamical degrees) of the map. However, only for very special maps which ideas from dynamics of diffeomorphisms of compact Riemann manifolds are applicable so far to the study of meromorphic maps. For these special meromorphic maps, one can construct very special invariant measures $\mu$ with no mass on proper analytic sets, which can be pushed forward. (Here, we recall this pushforward for the reader's convenience. Let $\mu$ be a probability measure with no mass on proper analytic subsets of $X$. Then we define $f_*(\mu )$ as the extension by zero of the probability measure $(f|_{X\backslash I(f)})_*(\mu )$, where $I(f)$ is the indeterminacy set of $f$ and hence $f$ is holomorphic on $X\backslash I(f)$.) For the majority of meromorphic maps, however, there is no obvious such special invariant measures, and one faces difficulty in constructing interesting invariant measures by Cesaro's average as above, since the cluster points of these measures (even if the starting measure $\mu _0$ has no mass on proper analytic subsets) are not guaranteed to have no mass on proper analytic subsets. The problem is then that we really do not know how to pushforward, in a reasonable and stay in the space of measures, a measure with support in the indeterminacy set $I(f)$ of $f$. For example, if $x_0\in I(f)$, then typically its image under the map $f$ will be of positive dimension, and there is no reasonable way to define the pushforward $f_*(\delta _{x_0})$, of the Dirac measure at $x_0$, as a {\bf measure}. Still, we can ask the following questions: Can we define instead the pushforward $f_*(\delta _{x_0})$ as something else more general than a measure? More importantly, can we hope to obtain some analog of the fundamental results mentioned above for all meromorphic maps? 

Besides the construction using Cesaro's average as above, when $X$ is K\"ahler, there is one other approach of constructing invariant measures, very related to special properties of compact K\"ahler manifolds, by intersecting dynamically interesting special positive closed currents (so-called Green's currents).  Intersection of positive closed currents, in particular those of bidegrees $(1,1)$, is an interesting topic itself with many applications in complex geometry and complex dynamics, and has been intensively studied. The methods employed so far by most researchers in this topic are local in nature, and the resulting intersections are supposed to be positive measures. These local methods also usually provide answers which are not compatible with intersection in cohomology, and the latter is a consideration one needs to take into account in order for the definition to be meaningful. However, again here not much is known about what to do if the currents to intersect are too singular. For example, is there any meaning to assign to self-intersection of the current of integration on a line $L\subset \mathbb{P}^2$?  For an even more interesting example, is there any meaning to assign  to self-intersection of the current of integration on a curve $C$ in a compact K\"ahler surface whose self-intersection in cohomology is $\{C\}.\{C\}=-1$? 

 One key question is that for dynamics of dominant meromorphic maps, should we take more consideration of the indeterminacy set of the map as well as of its iterates (this latter point has been so far not very much pursued by the works in the literature)? Likewise, in considering intersection of positive closed currents, should we  look closer at the singular parts of the currents involved, instead of throwing them away (as in the local approaches mentioned above)? 

All the maps mentioned above belong to a more general class of maps of the form: $f:U\rightarrow X$, where $X$ is a compact metric space, $U\subset X$ is open dense, and $f$ is a continuous map. Some natural examples of this kind of maps are transcendental maps $f:\mathbb{C}\rightarrow \mathbb{C}$, such as $f(z)=e^z$, where we choose $U=\mathbb{C}$ and $X=\mathbb{P}^1$ the preferable compactification of $U$. A generalisation of the latter maps are continuous maps $f:U\rightarrow U$, where $U$ is a non-compact metric space with a preferable compactification $X$.  In this paper, we give applications of strong submeasures, a classical but largely overlooked notion, to the dynamics of maps of these maps. Submeasures can also be applied to the question of intersection of positive closed currents on compact K\"ahler manifolds, to which we will address in a separate paper. 

To ease the presentation of the paper, we formally give a definition of the above class of maps. 

\begin{definition} Let $X$ and $Y$ be topological spaces. By an {\bf open-dense defined} map between $X$ and $Y$, we mean that there is an open dense subset $U$ of $X$ and a map $f:U\rightarrow Y$. We emphasise that $f$ does not need to be defined on the whole of $X$.  We denote such a map as $f:X\dashrightarrow Y$. 

If the map $f:U\rightarrow Y$ above is moreover continuous, then we say that $f$ is a continuous open-dense defined map between $X$ and $Y$.  In this case, we denote by OpenDom(f) the largest such subset $U$, and by OpenIm(f) the image $f(OpenDom(f))$. We also define the indeterminacy set of $f$ to be $I(f)=X\backslash OpenDom(f)$.  
\label{OpenDenseDefinedMapsDefinition}\end{definition}

We next define strong submeasures in detail. Let $X$ be a compact metric space. Denote by $\varphi \in C^0(X)$ the sup-norm $||\varphi ||_{L^{\infty}}=\sup _{x\in X}|\varphi (x)|$. We recall that a functional $\mu :C^0(X)\rightarrow \mathbb{R}$ is {\bf sub-linear} if $\mu (\varphi _1+\varphi _2)\leq \mu (\varphi _1)+\mu (\varphi _2)$ and $\mu (\lambda \varphi )=\lambda \mu (\varphi )$ for $\varphi _1,\varphi _2,\varphi \in C^0(X)$ and a non-negative constant $\lambda$. A {\bf strong submeasure} is then simply a sub-linear functional $\mu :C^0(X)\rightarrow \mathbb{R}$ which is also {\bf bounded}, that is there is a constant $C>0$ so that for all $\varphi \in C^0(X)$ we have $|\mu (\varphi )|\leq C||\varphi ||_{L^{\infty}}$. The least such constant $C$ is called the norm of $\mu$ and is denoted by $||\mu ||$. A strong submeasure $\mu$ is {\bf positive} if it is non-decreasing, that is for all $\varphi _1\geq \varphi _2$ we have $\mu (\varphi _1)\geq \mu (\varphi _2)$. It is easy to check that a strong submeasure is Lipschitz continuous $|\mu (\varphi _1)-\mu (\varphi _2)|\leq ||\mu ||\times ||\varphi _1-\varphi _2||_{L^{\infty}}$, and convex $\mu (t_1\varphi _1+t_2\varphi _2)\leq t_1\mu (\varphi _1)+t_2\mu (\varphi _2)$ for $t_1,t_2\geq 0$. We denote by $SM(X)$ the set of all strong submeasures on $X$, and $SM^+(X)$ the set of all positive strong submeasures on $X$. 

By Riesz representation theorem (see \cite{rudin}), on $X$ a measure of bounded mass  $\mu$ is the same as a {\bf linear} operator $\mu: C^0(X)\rightarrow \mathbb{R}$, that is $\mu (\lambda _1\varphi _1+\lambda _2\varphi _2)=\lambda _1\mu (\varphi _1)+\lambda _2(\varphi _2)$ for all $\varphi _1,\varphi _2\in C^0(X)$ and constants $\lambda _1,\lambda _2$; which is {\bf bounded}: $|\mu (\varphi )|\leq C||\varphi ||_{L^{\infty}}$ for a constant $C$ independent of $\varphi \in C^0(X)$, and {\bf positive}: $\mu (\varphi )\geq 0$ whenever $\varphi \in C^0(X)$ is non-negative. Note that we can then choose $C=\mu (1)$ (the mass of the measure $\mu $) and by linearity the positivity is the same as having 
\begin{equation}
\mu (\varphi _1)\geq \mu (\varphi _2)
\label{EquationPositivity}\end{equation} 
for all $\varphi _1,\varphi _2\in C^0(X)$ satisfying $\varphi _1\geq \varphi _2$.  For later reference, we denote by $M(X)$ the set of signed measures on $X$ and by $M^+(X)$ the set of positive measures on $X$. Note that $M(X)\subset SM(X)$ and $M^+(X)\subset SM^+(X)$. 

The {\bf first main idea} of this paper, to deal with pushforward of a positive strong submeasure, is as follows. If $U\subset X$ is an open dense set as above, then for any {\bf bounded} continuous function $\psi :U\rightarrow \mathbb{R}$, there is a canonical way to extend it to a bounded upper-semicontinuous function $E(g):X\rightarrow \mathbb{R}$. Now, if $f:U\rightarrow Y$ is a continuous function - where $Y$ is an another compact metric space - and $\mu$ is a positive finite Borel measure on $X$, then we can define a {\bf push-forward} $f_*(\mu )$ as a positive strong submeasure on $X$ in the following manner. If $\varphi :X\rightarrow \mathbb{R}$ is a continuous function, then $f^*(\varphi ): U\rightarrow \mathbb{R}$ is a bounded continuous function, and hence we have the canonical extension $E(f^*(\varphi ))$ which is a bounded upper-semicontinuous function on $X$. Then we define
\begin{eqnarray*}
(f_*\mu)(\varphi ):= \mu (E(f^*(\varphi ))).  
\end{eqnarray*}       
Details will be given in Section 2 where we show that the RHS of this definition gives rise to a strong positive submeasure and does not depend on the choice of $U$. Hence, this operator is well-defined for continuous  open-dense defined maps. The same idea can be applied to define more generally pushforward of positive strong submeasures. 

When $f:X\dashrightarrow X$ is a meromorphic map of a compact complex variety, then we can choose $U$ to be any open dense set of $X$ on which $f$ is a genuinely holomorphic function. In this case, because of Hironaka's resolution of singularities, one usually prefers to work with a desingularity of the graph of $f$. This is helpful in defining operations such as pulllback or pushforward of smooth closed forms, where the resulting is currents.  It will be shown that the definition given above for pushforward of measures can also be done using these desingularities of graphs, and hence illustrates that our definition is reasonable. 

When the map $f:U\rightarrow Y$ as above has dense image and is a covering map onto its image and of finite degree $d$ (for example, if this map is induced from a dominant meromorphic map between two compact complex spaces of the same dimension, and when $U$ is chosen appropriately), then we can also define similarly the {\bf pullback} of a positive strong submeasure. In Theorem \ref{TheoremSubmeasurePushforwardMeromorphic} we will prove some fundamental properties of these pushforward and pullback operators on positive strong submeasures. We extract here some most interesting properties. 

\begin{theorem} Let $X,Y$ be compact metric spaces, $U\subset X$ a dense open set, and  $f:X\dashrightarrow Y$ be  a continuous open-dense defined map. 

i) If $\mu _n\in SM^+(X)$ weakly converges to $\mu$,  and $\nu$ is a cluster point of $f_*(\mu _n)$, then $\nu \leq f_*(\mu )$. If $f$ is holomorphic, then $\lim _{n\rightarrow\infty}f_*(\mu _n)=f_*(\mu )$. 

ii) For any positive strong submeasure $\mu$, we have $f_*(\mu )=\sup _{\chi \in \mathcal{G}(\mu)}f_*(\chi )$, where $\mathcal{G}(\mu )=\{\chi :$ $\chi $ is a measure and $\chi \leq \mu \}$. 

iii) For positive strong submeasures $\mu _1,\mu _2$, we have $f_*(\mu _1+\mu _2)\geq f_*(\mu _1)+f_*(\mu _2)$. 

\label{TheoremSubmeasurePushforwardMeromorphicShortVersion}\end{theorem}

 In Example 2 in Section 2.2, we will show that strict inequality can happen in part i) in general. It also shows that, in contrast to the case of a continuous map - see Section 4, part ii) does not hold in general if we replace $\mathcal{G}(\mu )$ by a smaller set $\mathcal{G}$ (still satisfying $\mu =\sup _{\chi \in \mathcal{G}}\chi$). We also remark that several results in Theorem \ref{TheoremSubmeasurePushforwardMeromorphic} (such as parts 1,2,3) can be extended easily to meromorphic correspondences. 

Before going further, let us calculate explicitly one simple but interesting example. 

{\bf Example 1.} Let $\pi :Y\rightarrow X$ be the blowup of $X$ at a point $p$, and $V\subset Y$ the exceptional divisor. Let $\delta _p$ be the Dirac measure at $p$. Then for any continuous function $\varphi $ on $Y$, we have
\begin{eqnarray*}
\pi ^*(\delta _p)(\varphi )=\max _{y\in V}\varphi (y).
\end{eqnarray*}
Therefore, $\pi ^*(\delta _p)$ is {\bf not} a measure. In particular, if $A\subset Y$ is a closed set then $\pi ^*(\delta _p) (A)=\inf _{\varphi \in C^0(X,\geq 1_A)}\pi ^*(\delta _p)(\varphi )$ is $\delta _p(\pi (A\cap Y))$.   

\begin{proof}[Proof of Example 1] By definition 
\begin{eqnarray*}
\pi ^*(\delta _p)(\varphi )=\inf _{\psi \in C^0(Y,\geq \pi _*(\varphi ))}\delta _p(\psi )=\inf _{\psi \in C^0(Y,\geq \pi _*(\varphi ))}\psi (p). 
\end{eqnarray*}
Since $\pi :Y\backslash V\rightarrow X\backslash \{p\}$ is an isomorphism, it is easy to check that $\pi _*(\varphi )(p)=\max _{y\in V}\varphi (y)$. Therefore, for any $\psi \in C^0(Y,\geq \pi _*(\varphi ))$, we have $\psi (p)\geq \max _{y\in V}\varphi (y)$. Hence by definition $\pi ^*(\delta _p)(\varphi )\geq \max _{y\in V}\varphi (y)$. On the other hand, for any $\epsilon >0$, choose a small neighborhood $U_{\epsilon}$ of $p$ so that 
\begin{eqnarray*}
\sup _{y\in \pi ^{-1}(U_{\epsilon })}\varphi (y)\leq \epsilon + \max _{y\in V}\varphi (y). 
\end{eqnarray*}
It follows that $\sup _{U_{\epsilon }}\pi _*(\varphi )\leq \epsilon + \max _{y\in V}\varphi (y)$. Since $\pi _*(\varphi )$ is continuous on $X\backslash \{p\}$, it follows by elementary set theoretic topology that we can find a continuous function $\psi $ on $X$ so that $\psi \geq \pi _*(\varphi )$ and $\sup _{U_{\epsilon}}\psi \leq \epsilon + \max _{y\in V}\varphi (y)$. It follows that $\pi ^*(\delta _p)(\varphi )\leq \epsilon + \max _{y\in V}\varphi (y)$. Since $\epsilon $ is an arbitrary positive number, we conclude from the above discussion that $\pi ^*(\delta _p)(\varphi )=\max _{y\in V}\varphi (y)$. Similarly, we can show that $\pi ^*(\delta _p) (A)=\delta _p(\pi (A\cap Y))$. 
\end{proof}

Given $f:X\dashrightarrow Y$ a dominant meromorphic map between compact complex varieties and $\mu $ a positive strong submeasure on $X$. Without loss of generality, we can assume that $X$ is smooth, by using Hironaka's resolution of singularities. By part ii) of Theorem \ref{TheoremSubmeasurePushforwardMeromorphicShortVersion} and its proof, to describe $f_*(\mu )$ it suffices to describe $\pi ^*(\mu )$ where $\pi :Z\rightarrow X$ is a blowup at a smooth centre and $\mu$ is a measure. The following result addresses this question. 

\begin{theorem}
Let $\pi :Z\rightarrow X$ be the blowup of $X$ at an irreducible smooth subvariety $A\subset X$. Let $\varphi \in C^0(Z)$.  Let $\mu$ be a positive measure on $X$, and decompose $\mu =\mu _1+\mu _2$ where $\mu _1$ has no mass on $A$ and $\mu _2$ has support on $A$. Then $\pi ^*(\mu _1)$ is a positive measure on $Z$, $\pi _*(\varphi )|_A$ is continuous, and we have
\begin{eqnarray*}
\pi ^*(\mu )(\varphi )=\pi ^*(\mu _1)(\varphi )+\mu _2(\pi _*(\varphi )|_A).
\end{eqnarray*}
Moreover, an explicit choice of the collection $\mathcal{G}$ in part 2) of Theorem \ref{TheoremHahnBanach} for $\pi ^*(\mu )$ will be explicitly described in the proof. 
\label{TheoremPushforwardBlowup}\end{theorem}

The {\bf second main idea} of this paper, to deal with invariant positive strong submeasures, is as follows. One can start from a positive strong submeasure $\mu _0$, and perform the Cesaro's process to the  pushforward iterates $(f_*)^n(\mu _0)$. The cluster points $\mu _{\infty}$ are now in general not invariant - as in the case of continuous selfmaps of compact metric spaces - the reason being that the pushforward of continuous open-dense defined selfmaps $f:X\dashrightarrow X$ on positive strong submeasures is not continuous, as mentioned in the previous paragraph. It turns out, however, there is a canonical way, using a type of min-max principle, to assign to any such cluster point $\mu _{\infty}$ an invariant positive strong submeasure. Here is the main result. 

\begin{theorem} Let $f:X\dashrightarrow X$ be a continuous open-dense defined map on a compact metric space $X$. Let $0\not= \mu _0\in SM^+(X)$.

1) Let $\mu _0$ be a positive strong submeasure on $X$, and $\mu$ is a cluster point of Cesaro's averages $\frac{1}{n}\sum _{j=0}^n(f_*)^j(\mu _0)$. Then $f_*(\mu )\geq \mu$.


2) If $f_*(\mu _0)\leq \mu _0$, then the set  $\{\mu \in SM^+(X):~ \mu \leq \mu _0,~f_*(\mu )= \mu \}$ is non-empty and has a largest element, denoted by $Inv(\leq \mu _0)$. Moreover, $Inv(\leq \mu _0)=\lim _{n\rightarrow\infty}(f_*)^n(\mu _0)$. 

3) If $f_*(\mu _0)\geq \mu _0$, then the set $\{\mu \in SM^+(X):~ \mu \geq \mu _0,~f_*(\mu )= \mu \}$ is non-empty and has a smallest element, denoted by $Inv(\geq \mu _0)$. 


\label{TheoremInvariantMeasuresShortVersion}
\end{theorem}

The {\bf third and last main idea} of this paper, to deal with definitions of entropy, is as follows. Assume that $f:X\dashrightarrow X$ is a continuous open-dense defined selfmap. To define topological entropy for a continuous open-dense defined selfmap $f:X\dashrightarrow X$, we will follow the idea by S. Friedland. We will look at the closure $\Gamma _{f,\infty}$ of the graph of $f:OpenDom(f)\rightarrow X$ in the compact metric space $X^{\mathbb{N}}$. To be able to proceed as such, we need stronger properties of $f$ in order to be able to compose $f$ with itself any finite number of times, in analogy to holomorphic maps or dominant meromorphic maps. We formalise this notion next. 
\begin{definition}
Let $f:X\dashrightarrow Y$ be a continuous open-dense defined map. We say that $f$ is good with respect to iterates if the set $\Omega _{f,\infty}:= \{x\in OpenDom(f):$ $f^n(x)\notin I(f)$ for all $n\in \mathbb{N} \}$ is dense in $X$ and $I_{\infty}(f)=X\backslash \Omega _{f,\infty}$ is nowhere dense. 
\label{DefinitionGoodIterateMaps}\end{definition}
Note that this notion applies to dominant meromorphic selfmaps of compact complex varieties, as well as to continuous maps $f:U\rightarrow U$ where $U\subset X$ is open dense. 

In this case, we can define the infinity graph $\Gamma _{f,\infty}$ as the closure in $X^{\mathbb{N}}$ - with product topology - of $\{(x,f(x),f^2(x),f^{3}(x),\ldots )$ $:~x\in \Omega _{f,\infty}\}$. By Tikhonov's theorem, $X^{\mathbb{N}}$ is compact Hausdorff, and moreover is a compact metric space with the metric
\begin{eqnarray*}
d((x_1,x_2,\ldots ), (y_1,y_2,\ldots )):=\sum _{i=1}^{\infty} d(x_i,y_i), 
\end{eqnarray*}   
where $d(x_i,y_i)$ is the given metric on $X$. Hence $\Gamma _{f,\infty}$ is itself a compact metric space. On $X^{\mathbb{N}}$ there is a natural shifting map $\sigma (x_1,x_2,x_3,\ldots )=(x_2,x_3,\ldots )$, which is continuous. It is easy to check that $\sigma (\Gamma _{f,\infty})=\Gamma _{f,\infty}$, and we denote $\sigma _f:=\sigma |_{\Gamma _{f,\infty}}:~\Gamma _{f,\infty}\rightarrow \Gamma _{f,\infty}$.  Then the topological entropy of $f$ is defined as the topological entropy of $\sigma _f$. We denote this topological entropy $h_{top}(f)$, or when we want to emphasise the role of the compactification $X$ of $U$, we denote by $h_{top, X}(f)$. To define measure theoretic entropies, we note first that there is difficulties if we adapt directly the definition from continuous dynamics (defined that way, the measure theoretic entropies of an invariant positive strong submeasure can be infinity even for maps whose topological entropy is finite). To resolve this, we relate invariant positive strong submeasures on $X$ with invariant positive measures on $\Gamma _{f,\infty}$, through $\pi _1:\Gamma _{f,\infty}\rightarrow X$ the projection to the first coordinate, that is $\pi _1(x_1,x_2,x_3,\ldots )=x_1$.  This way allows us also to prove a version of the Variational Principle. In the case of measures with no mass on $I(f)$, we have the following result. 

\begin{proposition}
Let $f:X\dashrightarrow X$ be a continuous open-dense defined map of a compact metric space which is good with respect to iterates. Let $\mu$ be an invariant probability measure of $f$ with no mass on $I_{\infty}(f)$. Then there exists a unique measure $\hat{\mu}$ on $\Gamma _{f,\infty}$ so that $(\pi _1)_*(\hat{\mu})\leq \mu$, $||\hat{\mu}||=1$ and $(\phi _f)_*(\hat{\mu})=\hat{\mu}$. Moreover,
\begin{eqnarray*}
h_{\mu}(f)=h_{\hat{\mu}}(\phi _f). 
\end{eqnarray*}
 \label{PropositionMotivationMeasureEntropy}\end{proposition}

\begin{remarks} 

A related work has been given in \cite{truong3} when we defined topological entropy for the case of $f:U\rightarrow U$ for general non-compact topological spaces and not just metric spaces. Since a general topological space may not have a compactification, \cite{truong3} considered all compactifications of \'etale coverings of $f$. The current work can be used to strengthen the results in \cite{truong3}, in that we can now consider, besides topological entropy, other basic ergodic properties such as invariant submeasures and measure-theoretic entropy, as well as can deal with more general maps.      
\end{remarks}

The plan of this paper is as follows. In Section 2 we collect some basic properties of strong submeasures. In particular, we will define and prove properties of pushforward and pullback of positive strong submeasures for open-dense defined maps. In Section 3, we prove existence and some properties of invariant positive strong submeasures. In Section 4, we define and prove properties of topological and measure theoretic entropies for invariant positive strong submeasures, including a version of the Variational Principle. In the last section, we  give some applications to transcendental maps on $\mathbb{C}$ and $\mathbb{C}^2$, as well as to dynamics of dominant meromorphic maps of compact K\"ahler surfaces.  

{\bf Acknowledgments.} This paper is extracted and developed from the more dynamical part of our preprint \cite{truong4}. The part about least negative intersection of positive closed $(1,1)$ currents will be separated into another paper. 
 
\section{Strong submeasures} In this section we collect some basic properties of strong submeasures on a compact metric space $X$ which are needed to establish basic ergodic propeties, and the pushforward of them by continuous maps $f:U\rightarrow Y$, where $U\subset X$ is a dense open subset and $Y$ is another compact metric space. In case the map $f$ is a proper covering map of finite degree to its image $f(U)$, we can also define a pullback operator on positive strong submeasures.   
 
\subsection{Strong submeasures}  Let $X$ be a compact metric space. We recall the notations from the introduction: $M(X)$ the set of signed measures on $X$, $M^+(X)$ the set of positive measures on $X$, $SM(X)$ the set of strong submeasures on $X$ and $SM^+(X)$ the set of positive strong submeasures on $X$.

By a simple application of Hahn-Banach's extension theorem (see \cite{rudin2}) and Riesz representation theorem (see \cite{rudin}), we have the following characterisation, whose proof is left out, of strong submeasures and positive strong submeasures. 
\begin{theorem} Let $X$ be a compact metric space, and $\mu :C^0(X)\rightarrow \mathbb{R}$ an operator. 

1) $\mu $ is a strong submeasure if and only if there is a non-empty collection $\mathcal{G}$ of signed measures $\chi =\chi ^+-\chi ^-$ where $\chi ^{\pm}$ are measures on $X$ so that $\sup _{\chi =\chi ^{+}-\chi ^-\in \mathcal{G}}\chi ^{\pm}(1)<\infty$, and: 
\begin{equation}
\mu (\varphi )=\sup _{\chi \in \mathcal{G}}\chi (\varphi ), 
\label{EquationExampleSubmeasure}\end{equation}
for all continuous functions $\varphi$.

2) $\mu$ is a positive strong submeasure if and only if there is a non-empty collection $\mathcal{G}$ of (positive) measures on $X$ so that $\sup _{\chi \in \mathcal{G}}\chi (1)<\infty$, and:
\begin{equation}
\mu (\varphi )=\sup _{\chi \in \mathcal{G}}\chi (\varphi ),
\label{EquationExamplePositiveSubmeasure}\end{equation}
for all continuous functions $\varphi$. 
\label{TheoremHahnBanach}\end{theorem}

The next paragraphs discuss the natural topology on the space of strong submeasures. 

\begin{definition} We say that a sequence $\mu _1,\mu _2,\ldots \in SM(X)$ weakly converges to $\mu \in SM(X)$ if $\sup _n||\mu _n||<\infty$ and 
\begin{equation}
\lim _{n\rightarrow \infty}\mu _n(\varphi )=\mu (\varphi )
\label{EquationWeakConvergence}\end{equation} 
for all $\varphi \in C^0(X)$. We use the notation $\mu _n\rightharpoonup \mu$ to denote that $\mu _n$ weakly converges to $\mu$. 
\end{definition}

If $\mu _1,\mu _2:~C^0(X)\rightarrow \mathbb{R}$, we define $\max \{\mu _1,\mu _2\}:~C^0(X)\rightarrow \mathbb{R}$ by the formula $\max \{\mu _1,\mu _2\}(\varphi )=\max \{\mu _1(\varphi ), \mu _2(\varphi )\}$. Theorem \ref{TheoremSubmeasureBasicProperty} shows that submeasures also have properties similar to measures, such as weak compactness.  

For later use, we recall that  for a compact subset $A\subset X$, we have (\cite{doob}):
\begin{equation}
\mu (A)=\inf _{\phi \in C^0(X,\geq 1_A)}\mu (\phi ). 
\label{EquationMeasureClosedSet}\end{equation}
Here $1_A:X\rightarrow \{0,1\}$ is the characteristic function of $A$, that is $1_A(x)=1$ if $x\in A$ and $=0$ otherwise, $C^0(X)$ is the space of continuous functions from $X$ into $\mathbb{R}$, and for any bounded function $g:X\rightarrow \mathbb{R}$ we use the notations 
\begin{equation}
C^0(X, \geq g)=\{\phi \in C^0(X):~\phi \geq g\}. 
\label{EquationContinuousFunctionsBiggerG}\end{equation}
Moreover, for any {\bf open} set $B\subset X$ we have (\cite{doob})
\begin{equation}
\mu (B)=\sup _{A~\mbox{compact}\subset B}\mu (A). 
\label{EquationMeasureOpenSet}\end{equation}

Like measures, {\bf positive} strong submeasures give rise naturally to {\bf set functions}.  On a compact metric space $X$,  recall that a function $g :X\rightarrow \mathbb{R}$ is upper-semicontinuous if for every $x\in X$
\begin{equation}
\limsup _{y\rightarrow x}g (y)\leq g (x).
\label{EquationUpperSemicontinuity}\end{equation}
For example, the characteristic function of a closed subset $A\subset X$ is upper-semicontinuous. By Baire's theorem \cite{baire}, if $g$ is a bounded upper-semicontinuous function on $X$ then the set $C^0(X,\geq g)$ is non-empty and moreover
\begin{eqnarray*}
g=\inf _{\varphi \in C^0(X,\geq g)}\varphi . 
\end{eqnarray*}
More precisely, there is a sequence of continuous functions $g_n$ on $X$ decreasing to $g$. Hence, if $\mu$ is a {\bf measure}, we have by Lebesgue and Levi's monotone convergence theorem in the integration theory that
\begin{eqnarray*}
\mu (g)=\lim _{n\rightarrow\infty}\mu (g_n)=\inf _{\varphi \in C^0(X,\geq g)}\mu (\varphi ). 
\end{eqnarray*}

Inspired by this and (\ref{EquationMeasureClosedSet}), if $\mu$ is an {\bf arbitrary} strong submeasure, we define for any upper-semicontinuous function $g$ on $X$ the value
\begin{equation}
E(\mu) (g):=\inf _{\varphi \in C^0(X,\geq g)}\mu (\varphi )\in [-\infty ,\infty ).
\label{EquationSubmeasureUpperSemicontinuous}\end{equation}
Then for a closed set $A\subset X$, we define $\mu  (A):=E(\mu )(1_{A})$ where $1_A$ is the characteristic function of $A$. If $\mu$ is positive, we always have $\mu (A)\geq 0$. Then, for an open subset $B\subset X$, following (\ref{EquationMeasureOpenSet}) we define $\mu (B):=\sup \{\mu (A):~A$ compact  $\subset B\}$. Denote by $BUS(X)$ the set of all bounded upper-semicontinuous functions on $X$.  Theorem \ref{TheoremSubmeasureUpperSemicontinuous} proves some basic properties of this operator, similar to those of submeasures. 

If we have a positive strong submeasure $\mu$, and define for any Borel set $A\subset X$ the number $\widetilde{\mu}(A)=\inf \{\mu (B):$ $B$ open, $A\subset B\}$, then we see easily from part 4) of Theorem \ref{TheoremSubmeasureUpperSemicontinuous} that: i) $\widetilde{\mu}(\emptyset )=0$, ii) $\widetilde{\mu}(A_1)\leq \widetilde{\mu}(A_2)$ for all Borel sets $A_1\subset A_2$ and iii) 
$\widetilde{\mu}(A_1\cup A_2)\leq \widetilde{\mu}(A_1)+\widetilde{\mu}(A_2)$. Such $\widetilde{\mu}$ are known in the literature as submeasures (see e.g. \cite{talagrand}), and hence it is justified to call our objects $\mu$ positive strong submeasures.  

We have the following basic properties of strong submeasures.  
 
  \begin{theorem} Let $X$ be a compact metric space. 

1) {\bf Weak-compactness.} Let $\mu _1,\mu _2,\ldots $ be a sequence in $SM(X)$ such that $\sup _{n}||\mu _n||<\infty$. Then there is a subsequence $\{\mu _{n(k)}\}_{k=1,2,\ldots }$ which weakly converges to some $\mu \in SM(X)$. If moreover $\mu _n\in SM^+(X)$, then so is $\mu$. 

2) If $\mu \in SM^+(X)$, then $||\mu ||=\max \{|\mu (1)|,|\mu (-1)|\}$. 

3) If $\mu _1,\mu _2\in SM(X)$ then $\max \{\mu _1,\mu _2\}$ and $\mu _1+\mu _2$ are also in $SM(X)$.  If $\mu _1,\mu _2\in SM^+(X)$ then $\max \{\mu _1,\mu _2\}$ and $\mu _1+\mu _2$ are also in $SM^+(X)$.
\label{TheoremSubmeasureBasicProperty}\end{theorem}
 \begin{proof}[Proof of Theorem \ref{TheoremSubmeasureBasicProperty}]

1) Since $X$ is a compact metric space, the space $C^0(X)$, equipped with the $L^{\infty}$ norm,  is separable. Therefore, there is a countable set $\varphi _1,\varphi _2,\ldots $ which is dense in $C^0(X)$. Because $\sup _{n}||\mu _n||=C<\infty$, for each $j$ the sequence $\{\mu _n(\varphi _j)\}_{n=1,2,\ldots }$ is bounded. Therefore, using the diagonal argument, we can find a subsequence $\{\mu _{n(k)}\}_{k=1,2,\ldots }$  so that for all $j$ the following limit exists:
\begin{eqnarray*}
\lim _{k\rightarrow\infty}\mu _n(\varphi _j)=: \mu (\varphi _j).
\end{eqnarray*}
As observed in the introduction, the fact that $\mu _n$ is sublinear and bounded implies that it is Lipschitz continuous: $|\mu _n(\varphi )-\mu _n(\psi )|\leq ||\mu _n||\times ||\varphi -\psi ||\leq C||\varphi -\psi ||$ for all $n$ and all $\varphi ,\psi \in C^0(X)$.  Then from the fact that $\{\varphi _j\}_{j=1,2,\ldots }$ is dense in $C^0(X)$, it follows that for all $\varphi \in C^0(X)$, the following limit exists: 
 \begin{eqnarray*}
\lim _{k\rightarrow\infty}\mu _n(\varphi )=: \mu (\varphi ).
\end{eqnarray*}
It is then easy to check that $\mu $ is also a strong submeasure, and if $\mu _n$ are all positive then so is $\mu$.

2) For any $\varphi \in C^0(X)$ we have $-||\varphi ||_{L^{\infty}}\leq \varphi \leq ||\varphi ||_{L^{\infty}}$. Therefore, since $\mu $ is positive, we have $\mu (-||\varphi ||_{L^{\infty}})\leq \mu (\varphi )\leq \mu (||\varphi ||_{L^{\infty}})$. By the sub-linearity of $\mu$, we have $\mu (-||\varphi ||_{L^{\infty}})=||\varphi ||_{L^{\infty}}\mu (-1)$ and $\mu (||\varphi ||_{L^{\infty}})=||\varphi ||_{L^{\infty}}\mu (1)$. Therefore, $|\mu (\varphi )|\leq ||\varphi ||_{L^{\infty}}\max \{|\varphi (-1)|,|\varphi (1)|\}$. Therefore, $||\mu ||\leq \max\{|\varphi (-1)|, |\varphi (1)|\}$. The reverse inequality follows from the fact that $||-1||_{L^{\infty}}=||1||_{L^{\infty}}=1$. 

3) This is obvious. 
\end{proof}

\begin{theorem} Let $X$ be a compact metric space and $\mu \in SM(X)$. Let $E(\mu ):BUS(X)\rightarrow [-\infty , \infty )$ be defined as in (\ref{EquationSubmeasureUpperSemicontinuous}).  Assume that $E(\mu )(0)$ is finite. We have: 

1) For all $\varphi \in BUS(X)$, the value $E(\mu )(\varphi )$ is finite. Moreover, $E(\mu )(0)=0$ and $E(\mu )(-1)\geq -\mu (1)$. 

2) {\bf Extension.} If $\mu$ is {\bf positive}, then for all $\varphi \in C^0(X)$ we have $E(\mu )(\varphi )=\mu (\varphi )$.  

3) Moreover, $E(\mu )$ satisfies the following properties 

 i) {\bf Sub-linearity.} $E(\mu) (\varphi _1+\varphi _2)\leq E(\mu )(\varphi _1)+E(\mu) (\varphi _2)$ and $E(\mu )(\lambda \varphi )=\lambda E(\mu )(\varphi )$ for $\varphi _1,\varphi _2,\varphi \in BUS(X)$ and a non-negative constant $\lambda$. 

ii ) {\bf Positivity.} $E(\mu )(\varphi _1)\geq E(\mu )(\varphi _2)$ for all $\varphi _1,\varphi _2\in BUS(X)$ satisfying $\varphi _1\geq \varphi _2$. 

iii) {\bf Boundedness.} There is a constant $C>0$ so that for all $\varphi \in BUS(X)$ we have $|E(\mu )(\varphi )|\leq C||\varphi ||_{L^{\infty}}$. The least such constant $C$ is in fact $||\mu ||$.

4) If $A_1,A_2$ are closed subsets of $X$ then $\mu (A_1\cup A_2)\leq \mu (A_1)+\mu (A_2)$. Likewise, if $B_1,B_2$ are open subsets of $X$ then $\mu (B_1\cup B_2)\leq \mu (B_1)+\mu (B_2)$.  
\label{TheoremSubmeasureUpperSemicontinuous}\end{theorem}
\begin{proof}[Proof of Theorem \ref{TheoremSubmeasureUpperSemicontinuous}]
1) We first observe that for all $\varphi \in C^0(X,\geq 0)$ then $\mu (\varphi )\geq 0$. (Note that as observed in Section 1, this fact alone does not imply that $\mu$ is a positive strong submeasure.) In fact, otherwise, there would be $\varphi _0\in C^0(X,\geq 0)$ so that $\mu (\varphi _0)<0$. Then by the definition of $E(\mu )$ and sublinearity of $\mu$ we have
\begin{eqnarray*}
E(\mu )(0)\leq \inf _{n\in \mathbb{N}}\mu (n\varphi _0)=\inf _{n\in \mathbb{N}}n\mu (\varphi _0)=-\infty , 
\end{eqnarray*}
which is a contradiction with the assumption that $E(\mu )(0)$ is finite. 

Therefore, if $\varphi \in C^0(X,\geq 0)$, we obtain
\begin{eqnarray*}
0\leq \inf _{\psi \in C^0(X,\geq \varphi )}\mu (\psi )\leq \mu (\varphi ).
\end{eqnarray*}
Therefore, for these functions $\varphi$ we have $E(\mu )(\varphi )$ is a finite number. In particular, $0\leq E(\mu )(1)\leq \mu (1)$.  

Next, we observe that if $\varphi _1,\varphi _2\in BUS(X)$ such that either $E(\mu )(\varphi _1)$ or $E(\mu )(\varphi _2)$ is finite, then the proof of 3i) is still valid and gives $E(\mu )(\varphi _1+\varphi _2)\leq E(\mu )(\varphi _1)+E(\mu )(\varphi _2)$. Apply this sub-linearity to $\varphi _1=\varphi _2=0$ we obtain $E(\mu )(0)=E(\mu) (0+0)\leq 2E(\mu )(0)$, which implies $E(\mu )(0)\geq 0$. On the other hand, $E(\mu )(0)\leq \mu (0)=0$. Therefore, $E(\mu )(0)=0$. 
 
Since $E(\mu )(1)$ is finite, applying the above sub-linearity for $\varphi _1=1$ and $\varphi _2=-1$, we obtain $0=E(0)=E(\mu )(1+(-1))\leq E(\mu )(1)+E(\mu )(-1)$. Therefore, $E(\mu )(-1)\geq -E(\mu )(1)\geq -\mu (1)$.  

Finally, applying the proof of part 3iii) we deduce that for all $\varphi \in BUS(X)$, the number $E(\mu )(\varphi )$ is finite. 

2) Let $\varphi \in C^0(X)$, and choose any $\psi \in C^0(X,\geq \varphi )$. Since $\mu $ is positive, we have by definition that $\mu (\psi )\geq \mu (\varphi )$. Since $\varphi $ is itself contained in $C^0(X,\geq \varphi )$, it follows that 
\begin{eqnarray*}
E(\mu )(\varphi )=\inf _{\psi \in C^0(X,\geq \varphi )}\mu (\psi )=\mu (\varphi ). 
\end{eqnarray*}

3) Let $\varphi , \varphi _1,\varphi _2\in BUS(X)$. 

i) If $\psi _1\in C^0(X,\geq \varphi _1)$ and $\psi _2\in C^0(X,\geq \varphi _2)$ then $\psi _1+\psi _2\in C^0(X,\geq \varphi _1+\varphi _2)$. Therefore, by sub-linearity of $\mu$:
\begin{eqnarray*}
E(\mu )(\varphi _1+\varphi _2)=\inf _{\psi \in C^0(X,\geq \varphi _1+\varphi _2)}\mu (\psi )\leq \mu (\psi _1 + \psi _2)\leq \mu (\psi _1)+\mu (\psi _2). 
\end{eqnarray*}
We can choose $\psi _1$ and $\psi _2$ so that $\mu (\psi _1)$ is arbitrarily close to $E(\mu )(\varphi _1)$ and $\mu (\psi _2)$ is arbitrarily close to $E(\mu )(\varphi _2)$, and from that obtain the desired conclusion $E(\mu )(\varphi _1+\varphi _2)\leq E(\mu )(\varphi _1)+E(\mu )(\varphi _2)$. The other part of i) is easy to check. 

ii) If $\varphi _1\geq \varphi _2$ then $C^0(X,\geq \varphi _1)\subset C^0(X,\geq \varphi _2)$. From this the conclusion follows. 

iii)  We observe that we can find $\psi \in C^0(X,\geq \varphi )$ so that $||\psi ||_{L^{\infty}}=||\varphi ||_{L^{\infty}}$, simply by defining $\psi =\max\{\min \{\psi _0, ||\varphi ||_{L^{\infty}}\},-||\varphi ||_{L^{\infty}}\} $ for any $\psi _0\in C^0(X,\geq \varphi )$. Then $$E(\mu )(\varphi )\leq \mu (\psi )\leq ||\mu ||\times ||\psi ||_{L^{\infty}}=||\mu ||\times ||\varphi ||_{L^{\infty}}.$$ 
By the positivity of $E(\mu )$ in ii), we have $$E(\mu )(\varphi )\geq ||\varphi  ||_{L^{\infty}}E(\mu )(-1 ),$$
and hence $|E(\mu )(\varphi )|\leq \max \{|E(\mu )(-1)|,||\mu ||\}=||\mu ||$. In the last equality we used that 1) and positivity imply $-\mu (1)\leq E(\mu ) (-1)\leq E(\mu )(0)=0$.   

4) By definition we have for closed subsets $A_1,A_2\subset X$ 
\begin{eqnarray*}
\mu (A)= E(\mu )(1_{A_1\cup A_2})\leq E(\mu )(1_{A_1}+1_{A_2})\leq E(\mu )(1_{A_1})+E(\mu )(1_{A_2})=\mu (A_1)+\mu (A_2). 
\end{eqnarray*}
In the first inequality we used $1_{A_1\cup A_2}\leq 1_{A_1}+1_{A_2}$ and the positivity of $E(\mu )$. In the second inequality we used the sub-linearity of $E(\mu )$. 

If $B_1,B_2$ are open subsets of $X$ and $A\subset B_1\cup B_2$ is closed in $X$, then since $X$ is compact metric we can find closed subsets $A_1,A_2$ of $X$ so that $A_1\subset B_1$,  $A_2\subset B_2$ and $A_1\cup A_2=A$. To this end, for each $x\in A$, we choose an open ball $B(x,r_x)$ (in the given metric on $X$) where $r_x>0$ is chosen as follows: if $x\in B_1$ then the closed ball $\overline{B(x,r_x)}$ belongs to $B_1$, if $x\in B_2$ then the closed ball $\overline{B(x,r_x)}$ belongs to $B_2$, and if $x\in B_1\cap B_2$ then the closed ball $\overline{B(x,r_x)}$ belongs to $B_1\cap B_2$. Since $A$ is compact, there is a finite number of such balls covering $A$: $A\subset \bigcup _{i=1}^mB(x_i,r_i)$. Then the choice of $A_1=A\cap (\bigcup _{x_i\in B_1} \overline{B(x_i,r_i)})$  and $A_2=A\cap (\bigcup _{x_i\in B_2} \overline{B(x_i,r_i)})$ satisfies the requirement. Then from the above sub-linearity of $\mu $ for compact sets and the definition, we have also sub-linearity for open sets $\mu (B_1\cup B_2)\leq \mu (B_1)+\mu (B_2)$.   

\end{proof}

\subsection{Pushforward of positive strong submeasures} Throughout this subsection, we let $f:X\dashrightarrow Y$ be a continuous open-dense defined map, where $X$ and $Y$ are compact metric space. In this subsection we discuss several results concerning pushforward of positive strong submeasures on $X$ by $f$. 

First, we recall that to pushforward a measure by a {\bf continuous} map $g:X\rightarrow Y$ is the same as pullback continuous functions by $f$. In fact, the pushforward $f_*(\mu )$ of the measure $\mu$ is a measure which acts on continuous functions $\varphi :X\rightarrow \mathbb{R}$ by $f_*(\mu )(\varphi ):= \mu (f^*\varphi )$, where $f^*\varphi :X\rightarrow \mathbb{R}$ is the composition $\varphi \circ f$ and hence is also continuous. 

Now we consider the more general setting of a continuous open-dense defined map $f$ as before. We would like, as in the above paragraph, to define the pullback $f^*(\varphi )$ of a continuous function $\varphi :Y\rightarrow \mathbb{R}$. However, in general there is no reasonable way to define $f^*(\varphi )$ as a continuous function, because of the existence of  indeterminacy points. In general, $f^*(\varphi )$ is only continuous on $OpenDom(f)$, and cannot be extended to a continuous function on the whole of $X$. Fortunately, there is a canonical way to define the pullback $f^*(\varphi )$ as an upper-semicontinuous function, using only the value of $f^*(\varphi )$ on the open dense set where it is continuous. The key to this is the next result. 
\begin{proposition}
Let $X$ be a compact metric space, $U\subset X$ an open dense set, and $g:~U\rightarrow \mathbb{R}$ a bounded upper semicontinuous function. Define $E(g):X\rightarrow \mathbb{R}$ as follows: If $x\in U$ then $E(g)(x):=g(x)$, and if $x\in X\backslash U$ then 
\begin{eqnarray*}
E(g)(x):=\limsup _{y\in U,~y\rightarrow x}g(y).
\end{eqnarray*} 
Then 

1) $E(g)$ is a bounded upper-semicontinuous function, and $E(g)|_U=g$. In other words, $E(g)$ is a bounded upper-semicontinuous extension of $g$. 

2) If $g$ is continuous on $U$, $U_1\subset U$ is another open dense set of $X$ and $g_1=g|_{U_1}$, then $E(g_1)=E(g)$. 

3) Moreover, $E(g_1+g_2)\leq E(g_1)+E(g_2)$ for any $g_1,g_2:U\rightarrow \mathbb{R}$ bounded upper-semicontinuous functions. 
\label{PropositionUpperSemicontinuousExtension}\end{proposition}
{\bf Remark.} On the other hand, if $g$ is not continuous on $U$ then it is easy to construct examples for which the conclusion of part 2) in the proposition does not hold.  
\begin{proof}[Proof of Proposition \ref{PropositionUpperSemicontinuousExtension}] 

1) Since $g$ is bounded, it is clear that $E(g)$ is also bounded. By definition, it is clear that $E(g)|_U=g$. 

Next we show that $E(g)$ is upper-semicontinuous. If $x\in U$, then there is a small ball $B(x,r)\subset U$, and hence it can be seen that
\begin{eqnarray*}
\limsup _{y\in X\rightarrow x}E(g)(y)=\limsup _{y\in U\rightarrow x}E(g)(y)=\limsup _{y\in U\rightarrow x}g(y)\leq g(x),
\end{eqnarray*}
since $g$ is upper-semicontinuous on $U$. 

It remains to check that if $x\in X\backslash U$, and $x_n\in X\rightarrow x$ then 
\begin{eqnarray*}
\limsup _{n\rightarrow \infty}E(g)(x_n)\leq E(g)(x).
\end{eqnarray*}
To this end, choose $y_n\in U$ so that $d(y_n,x_n)\leq 1/n$ (here $d(.,.)$ is the metric on $X$) and $|g(y_n)-E(g)(x_n)|\leq 1/n$ for all $n$. Then $y_n\rightarrow x$, and hence
\begin{eqnarray*}
\limsup _{n\rightarrow\infty}E(g)(x_n)=\limsup _{n\rightarrow\infty}g(y_n)\leq E(g)(x). 
\end{eqnarray*} 

2) Using the assumption that $g$ is continuous on $U$ and $U_1\subset U$, we first check easily that $E(g_1)|_{U}=g=E(g)|_{U}$, and then the equality $E(g_1)=E(g)$ on the whole of $X$. 

3) Finally, note that if $g_1$ and $g_2$ are two upper-semicontinuous functions then $g_1+g_2$ is also upper-semicontinuous, and the inequality $E(g_1+g_2)\leq E(g_1)+E(g_2)$ follows from properties of limsup. 
\end{proof}
Let $\Gamma _f\subset X\times Y$ be the closure of the graph $\{(x,f(x)):~x\in OpenDom(f)\}$. Then, with the induced topology from $X\times Y$, $\Gamma _f$ is a compact metric space itself. We have two canonical projections $\pi _X,\pi _Y:~X\times Y\rightarrow X,Y$, whose restrictions to $\Gamma _f$ are denoted $\pi _{X,f},\pi _{Y,f}$. If we let $U\subset OpenDom(f)$ be any open dense subset, and by $V=\pi _{X,f}^{-1}(U)$, then $V$ is an open dense subset of $\Gamma _f$ and we see that $\pi _{V,f}:=\pi _{X,f}|_V:V\rightarrow U$ is a homeomorphism. 

\begin{definition} Using Proposition \ref{PropositionUpperSemicontinuousExtension}, we define $(\pi _{X,f})_*(\phi )$, where $\phi \in C^0(\Gamma _f)$, to be the upper-semicontinuous function  $E((\pi _{V,f})_*(\phi ))$ on $X$. We emphasise that it is {\bf globally defined} on the whole of $X$, and is not changed if we replace $V$ (or $U$) by one open dense subset of it. 
\end{definition}

We have finally a canonical definition of pullback of continuous functions by $f:X\dashrightarrow Y$. 

\begin{definition} Let $\varphi \in C^0(Y)$. We denote by $f^*(\varphi )$ the above upper-semicontinuous function $(\pi _{X,f})_*(\pi _{Y,f}^*(\phi ))$. 
\label{DefinitionPullbackContinuousFunctions} \end{definition}

 Note that when $X,Y$ are smooth complex varieties, $f:X\dashrightarrow Y$ a meromorphic map. and $\varphi $ is a continuous quasi-psh function on $Y$ (for example, if $\varphi$ is a $C^2$ function), then our definition of pullback $f^*(\varphi)$ above is the standard one in \cite{meo}. In this case, any continuous function $\varphi$ is a uniform limit of a sequence of $C^2$ functions. This observation and the following proposition, whose simple proof is left out, justify our Definition \ref{DefinitionPullbackContinuousFunctions}. 
 
\begin{proposition} Let $f:X\dashrightarrow Y$ be a continuous open-dense defined map as in the beginning of this section. Assume that ${\varphi _n}$ is a sequence of continuous functions on $Y $ uniformly converging to a continuous function $\varphi$. Then $\{f^*(\varphi _n)\}$ converges uniformly to $f^*(\varphi)$. 
\label{PropositionContinuityPullbackFunctions}\end{proposition} 

Using Proposition \ref{PropositionUpperSemicontinuousExtension} and the above upper-semicontinuous pushforward of functions,  we can finally define following Theorem \ref{TheoremSubmeasureUpperSemicontinuous}, the following pullback operator $\pi _{X,f}^*:SM^+(X)\rightarrow SM^+(\Gamma _f)$:
\begin{equation}
\pi _{X,f}^*(\mu )(\varphi ):=\inf _{\psi \in C^0(X,\geq (\pi _{X,f})_*(\varphi ))}\mu (\psi ). 
\label{EquationSubmeasurePullback}\end{equation}
Then, as in the case of continuous maps $g:X\rightarrow Y$ between compact metric spaces, we define $f_*(\mu )$, of a positive strong submeasure $\mu$, by the formula $f_*(\mu )=(\pi _{Y,f})_*\pi _{X,f}^*(\mu )$. 

\begin{definition} 
For convenience, we write here the final formula for pushforwarding a strong submeasure by our maps $f:X\dashrightarrow Y$:
\begin{equation}
f_*(\mu )(\varphi ):=\inf _{\psi \in C^0(X,\geq (\pi _{X,f})_*(\pi _{Y,f}^*(\varphi )))}\mu (\psi ). 
\label{EquationSubmeasurePushforwardMeromorphic}\end{equation}
\end{definition}

Here is the main result of this subsection.  In its proof we use pullback of positive strong submeasures by the projection $\pi $, via pushforward by its inverse $\pi ^{-1}$. For more general about pullback of positive strong submeasures, see the next subsection. 

\begin{theorem} Let $X,Y$ be compact metric spaces, and $f:X\dashrightarrow Y$ a continuous open-dense defined map. 

1) We have $f_*(SM^+(X))\subset SM^+(Y)$. Moreover, if $\mu \in SM^+(X)$, then $f_*(\mu )(\pm 1)=\mu (\pm 1)$. In particular, $||f_*(\mu )||=||\mu ||$. 

2) Assume that $g:Y\dashrightarrow Z$ is another continuous open-dense defined, and there is an open dense set $U\subset OpenDom(f)$ such that $f(U)\subset OpenDom(g)$.  For all $\mu \in SM^+(X)$ we have $g_*f_*(\mu )\geq (g\circ f)_*(\mu )$. If $f$ and $g$ are {\bf continuous} on the whole of $X$ and $Y$, then equality happens. 
 

3) If $\mu _n\in SM^+(X)$ weakly converges to $\mu$,  and $\nu$ is a cluster point of $f_*(\mu _n)$, then $\nu \leq f_*(\mu )$. If $f$ is continuous on the whole of $X$, then $\lim _{n\rightarrow\infty}f_*(\mu _n)=f_*(\mu )$.  

4) If $\mu$ is a positive measure without mass on $I(f)$, then $f_*(\mu )$ is the same as the usual definition. 

5) For any positive strong submeasure $\mu$, we have $f_*(\mu )=\sup _{\chi \in \mathcal{G}(\mu)}f_*(\chi )$, where $\mathcal{G}(\mu )=\{\chi :$ $\chi $ is a measure and $\chi \leq \mu \}$. 

6) For every positive strong submeasures $\mu _1,\mu _2$, we have $f_*(\mu _1+\mu _2)\geq f_*(\mu _1)+f_*(\mu _2)$. 
\label{TheoremSubmeasurePushforwardMeromorphic}\end{theorem}
\begin{proof}[Proof of Theorem \ref{TheoremSubmeasurePushforwardMeromorphic}]

1) Let $\mu \in SM^+(X)$, we will show that $f_*(\mu )\in SM^+(Y)$. Let $\varphi ,\varphi _1,\varphi _2\in C^0(Y)$ and $0\leq \lambda \in \mathbb{R}$. 

First, we show that $f_*(\mu )(\pm 1)=\mu (\pm 1)$. In fact, it follows from the definition that $(\pi _{X,f})_*\pi _{Y,f}^*(\pm 1)=\pm 1$, and hence \begin{eqnarray*}
f_*(\mu )(\pm 1)=\inf _{\psi \in C^0(X,\geq \pm 1)}\mu (\psi )=\mu (\pm 1). 
\end{eqnarray*}

Second, we show the positivity of $f_*(\mu )$. If $\varphi _1\geq \varphi _2$, then it can see from the definition that $(\pi _{X,f})_*\pi _{Y,f}^*( \varphi _1)\geq (\pi _{X,f})_*\pi _{Y,f}^*( \varphi _2)$. Therefore $C^0(X,\geq (\pi _{X,f})_*\pi _{Y,f}^*( \varphi _1))\subset C^0(X,\geq (\pi _{X,f})_*\pi _{Y,f}^*( \varphi _2))$, and hence it follows from definition that $f^*(\mu )(\varphi _1)\geq f^*(\mu )(\varphi _2)$. 

Next, we show that $f_*(\mu )$ is bounded and moreover $||f_*(\nu )||=\deg (f)||\mu ||$. By positivity of $f_*(\mu )$, we have $f_*(\mu )(-||\varphi ||_{L^{\infty}})\leq f_*(\mu )\leq f_*(\mu )(||\varphi ||_{L^{\infty}})$. Hence $f_*(\nu )$ is bounded, and we conclude by Theorem \ref{TheoremSubmeasureBasicProperty}. 

Finally, we show the sub-linearity. The equality $f_*(\mu )(\lambda \varphi )=\lambda f_*(\mu )(\varphi )$, for $\lambda \geq 0$, follows from the fact that $(\pi _{X,f})_*\pi _{Y,f}^*(\lambda \varphi )=\lambda (\pi _{X,f})_*\pi _{Y,f}^*(\varphi )$ and properties of infimum. We now prove that $f_*(\mu )(\varphi _1+\varphi _2)\leq f^*(\mu )(\varphi _1)+f^*(\mu )(\varphi _2)$. In fact, from Proposition \ref{PropositionUpperSemicontinuousExtension} we have $$(\pi _{X,f})_*\pi _{Y,f}^*(\varphi _1+\varphi _2)\leq (\pi _{X,f})_*\pi _{Y,f}^*( \varphi _1)+(\pi _{X,f})_*\pi _{Y,f}^*( \varphi _2),$$
and hence if $\psi _1\in C^0(X,\geq (\pi _{X,f})_*\pi _{Y,f}^*( \varphi _1))$ and $\psi _2\in C^0(X, \geq (\pi _{X,f})_*\pi _{Y,f}^*( \varphi _2))$ then $\psi _1+\psi _2\in C^0(X,\geq (\pi _{X,f})_*\pi _{Y,f}^*(\varphi _1+\varphi _2))$. Hence, by definition 
\begin{eqnarray*}
f_*(\mu )(\varphi _1+\varphi _2)\leq \mu (\psi _1+\psi _2)\leq \mu (\psi _1) + \mu (\psi _2). 
\end{eqnarray*}
In the second inequality we used the sub-linearity of $\mu$. If we choose $\psi _1$ and $\psi _2$ so that $\mu (\psi _1)$ is close to $f_*(\mu )(\varphi _1)$ and $\mu (\psi _2)$ is close to $f_*(\mu )(\varphi _2)$, then we see that $f_*(\mu )(\varphi _1+\varphi _2)\leq f_*(\mu )(\varphi _1)+f_*(\mu )(\varphi _2)$ as wanted. 

 2) By definition, we have
 \begin{eqnarray*}
 (g\circ f)_*(\mu )(\varphi )=\inf _{\psi \in C^0(X,\geq (g\circ f)^*(\varphi ))}\mu (\psi ). 
 \end{eqnarray*}
 Here we recall that $(g\circ f)^*(\varphi )$ is the upper-semicontinuos pullback of $\varphi $ by $g\circ f$. 
 
 On the other hand,
 \begin{eqnarray*}
g_*f_*(\mu )(\varphi )=\inf _{\psi _1\in C^0(Y,\geq g^*(\varphi ))}f_*(\mu )(\psi _1)=\inf _{\psi _1\in C^0(Y,\geq g^*(\varphi ))}\inf _{\psi _2\in C^0(X,\geq f^*(\psi _1))} \mu (\psi _2). 
 \end{eqnarray*}

 Then,it follows by the proof of part 2) of Proposition \ref{PropositionUpperSemicontinuousExtension} that whenever $\psi _1\in C^0(Y,\geq g^*(\varphi ))$  and $\psi _2\in C^0(Y,\geq f^*(\psi _1 ))$, then $\psi _2\in C^0(X, \geq (g\circ f)^*(\varphi ))$. From this, we get $g_*f_*(\varphi )(\nu )\geq (g\circ f)_*(\mu )(\varphi )$. 
 
When $f$ and $g$ are continuous on the whole $X$ and $Y$ and $\varphi $ is continuous, then $g^*(\varphi )$, $f^*(g^*(\varphi ))$ and $(g\circ f)^*(\varphi )$ are all continuous functions. Then using the positivity of $\mu$, we can easily see that $$(g\circ f)_*(\mu )(\varphi )=\mu (f^*g^*(\varphi ))=f_*(\mu )(g^*(\varphi ))=g_*f_*(\mu )(\varphi ).$$  



3)  It is enough to show the following: for all $\varphi \in C^0(Y)$ then 
\begin{eqnarray*}
\limsup _{n\rightarrow \infty}\inf _{\psi \in C^0(X,\geq f^*(\varphi ))}\mu _n(\psi )\leq \inf _{\psi \in C^0(X,\geq f^*(\varphi ))}\mu (\psi ). 
\end{eqnarray*}
If we choose $\psi _0\in  C^0(X,\geq f^*(\varphi ))$ so that $\mu (\psi _0)$ is close to $f_*(\mu )(\varphi )$, then from $\mu _n(\psi _0)\rightarrow \mu (\psi _0)$ we obtain the conclusion. 

If $f$ is continuous on the whole of $X$, then $f^*(\varphi )$ is itself a continuous function. Then it is easy to see  that $\lim _{n\rightarrow\infty}f_*(\mu _n)(\varphi )=f_*(\mu )(\varphi )$. 

4)  The upper-semicontinuous pullback $f^*(\varphi )$ of a function $\varphi \in C^0(Y)$ is continuous on the open set $U=X\backslash I(f)$. Therefore, by choosing a small open neighborhood $U_1$ of $I(f)$ and a partition of unity subordinate to $U$ and $U_1$, it is easy to find for any $U_2\subset\subset U$ a $\psi \in C^0(X,\geq f^*(\varphi ))$ so that $\psi |_{U_2}=f^*(\varphi )|_{U_2}$.  From this and the assumption that $\mu$ has no mass on $I(f)$,  the conclusion follows. 

5) As mentioned before the statement of the theorem, since $\pi _{X,f}:\Gamma _f\backslash \pi _{X,f}^{-1}(I(f))\rightarrow X\backslash I(f)$ is a homeomorphism, we can define the pullback of a positive strong submeasure on $X$ - as a positive strong submeasure on $\Gamma _f$ - as the pushforward of the continuous open-dense defined map $\pi _{X,f}^{-1}:X\dashrightarrow \Gamma _f$. We have for any positive strong submeasure $\mu$ that $f_*(\mu )=(\pi _{Y,f})_*\pi _{X,f} ^*(\mu )$. It is easy to check that the conclusion holds for $\pi _{Y,f}$, and hence to prove the result it suffices to prove that $\pi _{X,f}^*(\mu )=\sup _{\chi \in \mathcal{G}(\mu )}\pi ^*(\chi )$.  

Since $\pi _{X,f}^*(\mu )\geq \pi _{X,f}^*(\chi )$ for all $\chi \in \mathcal{G}(\mu )$, it follows that $\pi _{X,f}^*(\mu )\geq \sup _{\chi \in \mathcal{G}(\mu )}\pi _{X,f} ^*(\chi )$. Now we will prove the reverse inequality. To this end, it suffices to show that for any measure $\chi '\leq \pi ^*(\mu )$, there is a measure $\chi \leq \mu$ so that $\chi '\leq \pi ^*(\chi )$. 

We first show that $(\pi _{X,f})_*\pi _{X,f} ^*(\mu )=\mu $. In fact, if $\varphi \in C^0(X)$ then $\pi _{X,f}^*(\varphi )\in C^0(Z)$ and $\varphi =(\pi _{X,f})_*\pi _{X,f}^*(\varphi )$.  Hence, by definition 
\begin{eqnarray*}
(\pi _{X,f})_*\pi _{X,f}^*(\mu )(\varphi )=\pi _{X,f}^*(\mu )(\pi _{X,f}^*(\varphi ))=\mu ((\pi _{X,f})*\pi _{X,f}^*(\varphi ))=\mu (\varphi ). 
\end{eqnarray*} 
Hence $(\pi _{X,f})_*\pi _{X,f}^*(\mu )=\mu $ as wanted. 

Now if $\chi '$ is any measure on $Z$, then $\chi =(\pi _{X,f})_*(\chi ')$ is a measure on $X$. If moreover, $\chi '\leq \pi _{X,f}^*(\mu )$, then
\begin{eqnarray*}
\chi =(\pi _{X,f})_*(\chi ')\leq (\pi _{X,f})_*\pi _{X,f}^*(\mu )=\mu . 
\end{eqnarray*}
To conclude the proof, we will show that $\pi _{X,f}^*(\chi )\geq \chi '$. To this end, let $\varphi \in C^0(Z)$, we will show that $\pi _{X,f}^*(\chi )(\varphi )\geq \chi '(\varphi )$. By definition, the value of the positive strong submeasure $\pi _{X,f}^*(\chi )$ at $\varphi$ is defined as: $\pi _{X,f}^*(\chi )(\varphi )=\inf _{\psi \in C^0(X,\geq (\pi _{X,f})_*(\varphi ))}\chi (\psi )$, and since $\chi =(\pi _{X,f})_*(\chi ')$ the RHS is equal to $\inf _{\psi \in C^0(X,\geq (\pi _{X,f})_*(\varphi ))}\chi '(\pi _{X,f}^*(\psi ))\geq \chi '(\varphi )$. The latter follows from the fact that $\chi '$ is a positive strong submeasure and that for all $\psi \in C^0(X,\geq (\pi _{X,f})_*(\varphi ))$ we have $\pi _{X,f}^*(\psi )\geq \varphi $.   

6) Let $\varphi$ be a continuous function on $X$. We then have by part 5, using $\mathcal{G}(\mu _1)+\mathcal{G}(\mu _2)\subset \mathcal{G}(\mu _1+\mu _2)$, that
\begin{eqnarray*}
f_*(\mu _1+\mu _2)(\varphi )&=&\sup _{\nu \in \mathcal{G}(\mu _1+\mu _2)}\inf _{\psi \in C^0(\geq f^*(\varphi ))}\nu (\psi )\\
&\geq&\sup _{\nu _1\in \mathcal{G}(\mu _1),\nu _2\in \mathcal{G}(\mu _2)}\inf _{\psi \in C^0(\geq f^*(\varphi ))}(\nu _1+\nu _2) (\psi )\\
&\geq&\sup _{\nu _1\in \mathcal{G}(\mu _1),\nu _2\in \mathcal{G}(\mu _2)}[\inf _{\psi \in C^0(\geq f^*(\varphi ))}\nu _1 (\psi )+\inf _{\psi \in C^0(\geq f^*(\varphi ))}\nu _2 (\psi ) ]\\
&=&f_*(\mu _1)(\varphi )+f_*(\mu _2)(\varphi ).
\end{eqnarray*}
\end{proof}

{\bf Example 2.} Let $J:\mathbb{P}^2\dashrightarrow \mathbb{P}^2$ be the standard Cremona map given by $J[x_0:x_1:x_2]=[1/x_0:1/x_1:1/x_2]$. It is a birational map and is an involution: $J^2=$ the identity map. Let $e_0=[1:0:0]$, $e_1=[0:1:0]$ and $e_2=[0:0:1]$, and $\Sigma _i=\{x_i=0\}$ ($i=0,1,2$). Let $\pi :X\rightarrow \mathbb{P}^2$ be the blowup of $\mathbb{P}^2$ at $e_0,e_1$ and $e_2$, and let $E_0, E_1$ and $E_2$ be the corresponding exceptional divisors. Let $h=f\circ \pi :X\rightarrow \mathbb{P}^2$, then $h$ is a holomorphic map. Moreover, $\pi ^{-1}(e_0)=E_0$ and $h(E_0)=\Sigma _0$. More precisely, we have $h^{-1}(\Sigma _0\backslash \{e_1,e_2\})\subset E_0$. From this, we can compute, as in Example 1 in the introduction and proof of part 2) of Proposition \ref{PropositionUpperSemicontinuousExtension}, that for all $\varphi \in C^0(X)$ and for $\delta _{e_0}$ the Dirac measure at $e_0$
\begin{eqnarray*}
J_*(\delta _{e_0})(\varphi )=\sup _{\Sigma _0}\varphi . 
\end{eqnarray*}
It follows that $J_*(\delta _{e_0})\geq \max \{\delta _{e_1}, \delta _{e_2}\}$, where $\delta _{e_1}$ is the Dirac measure at $e_1$ and $\delta _{e_2}$ is the Dirac measure at $e_2$. Therefore, by the positivity of $J_*$ we obtain: 
\begin{eqnarray*}
J_*J_*(\delta _0 )(\varphi )\geq J_*(\max \{\delta _{e_1}, \delta _{e_2}\})(\varphi )\geq \max \{J_*(\delta _{e_1}(\varphi )), J_*(\delta _{e_2}(\varphi ))\}=\max \{\sup _{\Sigma _1}\varphi , \sup _{\Sigma _2}\varphi \}. 
\end{eqnarray*}
On the other hand, $J\circ J=$ the identity map, and hence $(J\circ J)_*(\delta _{e_0})=\delta _{e_0}$. Hence the inequality in part 2) of Theorem \ref{TheoremSubmeasurePushforwardMeromorphic} is strict in this case. 

This example also shows that the inequality in part 3) of Theorem \ref{TheoremSubmeasurePushforwardMeromorphic} is strict in general. In fact, let $\{p_n\}\subset X\backslash I(f)$ be a sequence converging to a point $p=e_0$ and $\{J(p_n)\}$ converges to a point $q\in \Sigma _0$. Let $\mu _n=\max \{\delta _{p_1},\ldots ,\delta _{p_n}\}$. It can be checked easily that $\mu _n$ is an increasing sequence of positive strong submeasures, with $J_*(\mu _n)=\max \{\delta _{J(p_1)},\ldots ,\delta _{J(p_n)}\}$ for all $n$. Then the weak convergence limit $\mu =\lim _{n\rightarrow\infty}\mu _n=\sup _n\delta _{p_n}$ exists. In particular, $\mu \geq \delta _{e_0}$, and hence from the above calculation we find $J_*(\mu )\geq \sup _{x\in \Sigma _0}\delta _x$. On the other hand, $\nu =\lim _{n\rightarrow \infty}J_*(\mu _n)=\sup _{n}\delta _{J(p_n)}$. It is clear that if $x\in \Sigma _0\backslash \{q\}$, then $\nu$ cannot be compared with $\delta _x$. Therefore, we have the strict inequality $J_*(\mu )>\nu$ in this case.  

If we choose a sequence of points $\{p_n\}_{n=1,2,\ldots }\subset X\backslash \{e_0,e_1,e_2\}$ converging to $e_0$ and such that $q_n=J(p_n)$ converges to a point $q_0\in \Sigma _0$, then it can be seen that for $\mu =\sup _n\delta _{p_n}$ we have $J_*(\mu )>\sup _nJ_*(\delta _{p_n})$. Hence part 5) of Theorem \ref{TheoremSubmeasurePushforwardMeromorphic} does not hold in general if we replace $\mathcal{G}(\mu )$ by the smaller set $\mathcal{G}=\{\delta _{p_n}\}_n$.

\subsection{Pullback of positive strong submeasures}

Let $f:X\dashrightarrow Y$ be a continuous open-dense defined map, so that there is an open dense subset $U\subset X$ such that $f(U)$ is open dense in $Y$ and  $f_U=f|_U:U\rightarrow f(U)$ is a proper covering map of finite degree. In this subsection we will define the pullback of positive strong submeasures on $Y$ for such maps. 

To this end, we note that for each  $\varphi \in C^0(X)$, the function 
\begin{eqnarray*}
(f|_U)_*(\varphi )(y)=\sum _{x\in (f|_U)^{-1}(y)}\varphi (x)
\end{eqnarray*}
is a continuous function on $f(U)$. Therefore, using Proposition \ref{PropositionUpperSemicontinuousExtension} we can define the following upper-semicontinuous function on $Y$:
\begin{eqnarray*}
f_*(\varphi ):=E((f|_U)_*(\varphi )).  
\end{eqnarray*}
Then, similarly to the previous subsection, we can define for $\nu \in SM^+(Y)$ and $\varphi \in C^0(X)$: 
\begin{eqnarray*}
f^*(\nu )(\varphi ):= \inf _{\psi \in C^0(Y,\geq f_*(\varphi ))}\nu (\psi ). 
\end{eqnarray*}

We have the following relation between pullback and pushforward, whose simple proof is not included. 
\begin{theorem}
Let $f:X\dashrightarrow Y$ be a continuous open-dense defined map. Assume that there is an open dense set $U\subset OpenDom(f)$ so that $f:U\rightarrow f(U)$ is a homeomorphism. Let $f^{-1}:Y\dashrightarrow X$ be the inverse open-dense defined map of $f$. Then for all  $\nu \in SM^+(Y)$ we have $f^*(\nu )=(f^{-1})_*(\nu )$.  
\label{TheoremPushforwardPullback}\end{theorem}

The proof of the following result is similar to that of Theorem \ref{TheoremSubmeasurePushforwardMeromorphic}.

\begin{theorem} Let $X,Y$ be compact metric spaces, and $f:X\dashrightarrow Y$ a continuous open-dense defined map so that there is  an open dense subset $U\subset X$ such that $f(U)$ is open dense in $Y$ and  $f_U=f|_U:U\rightarrow f(U)$ is a proper covering map of finite degree. Let $g: Y\dashrightarrow Z$ be another continuous open-dense defined map so that there is  an open dense subset $V\subset Y$ such that $g(V)$ is open dense in $Y$ and  $g_V=g|_V:V\rightarrow g(U)$ is a proper covering map of finite degree.. 

1) We have $f^*(SM^+(Y))\subset SM^+(X)$. Moreover, if $\nu \in SM^+(Y)$, then $f^*(\nu )(\pm 1)=\deg (f) \nu (\pm 1)$. In particular, $||f^*(\nu )||=\deg (f) ||\nu ||$. 

2) For all $\nu \in SM^+(Y)$ we have $f^*g^*(\mu )\geq (g\circ f)^*(\nu )$. If $f$ and $g$ are {\bf continuous} on the whole of $X$ and $Y$, and $g$ is homeomorphic on an open dense subset $V$ of $Y$, then equality happens. 
 
\label{TheoremSubmeasurePullbackMeromorphic}\end{theorem}

\section{Invariant positive strong submeasures}

In the previous section we defined for each continuous open-dense defined $f:X\dashrightarrow Y$ an operator $f_*:SM^+(X)\rightarrow SM^+(Y)$. Here we apply this to the case where $X=Y$ to produce invariant positive strong submeasures by combining Cesaro's averages with a min-max principle, as stated in Theorem  \ref{TheoremInvariantMeasuresShortVersion}. After that, we prove some further results for maps which are good with respect to iterates, see Definition \ref{DefinitionGoodIterateMaps}, to prepare for discussion about topological and measure theoretic entropy for such maps in the next section.   

We now discuss some properties of invariant positive strong submeasures. We note that in the case $f$ is a continuous map and $\mu$ is a measure, then $f_*(\mu)=\mu$ if and only if either $f_*(\mu)\geq \mu$ or $f_*(\mu )\leq \mu$. In the general case we are concerned here, the properties $f_*(\mu )=\mu$,  $f_*(\mu )\geq \mu$ and $f_*(\mu )\leq \mu$ are in general not the same. However, these properties are very much related. For example, it can be checked that if $\mu$ is a measure and $f_*(\mu)\leq \mu$, then $f_*(\mu)=\mu$. On the other hand, we will see that positive strong submeasures $\mu$ having the property that $f_*(\mu)\geq \mu$ appear very naturally in dynamics. If we apply Cesaro's average procedure for meromorphic maps, we obtain positive strong submeasures $\mu$ with $f_*(\mu)\geq \mu$. Likewise, if we have a positive strong submeasure $\hat{\mu}$ on $\Gamma _{f,\infty}$ which is $\phi _f$ - invariant, then $\mu =(\pi _1)_*\hat{\mu}$ satisfies $f_*(\mu)\geq \mu$. Using this, we obtain canonical invariant positive strong submeasures to those $\mu$ which satisfy either $f_*(\mu)\geq \mu$ or $f_*(\mu )\leq \mu$ in Theorem  \ref{TheoremInvariantMeasuresShortVersion}. We next detail proofs of these claims.   

 \begin{proof}[Proof of Theorem \ref{TheoremInvariantMeasuresShortVersion}]

1) We note that for $\mu _n=\frac{1}{n}\sum _{j=0}^n(f_*)^j(\mu _0)$, by part 6 of Theorem \ref{TheoremSubmeasurePushforwardMeromorphic} we have
\begin{eqnarray*}
f_*(\mu _n)-\mu _n&=&f_*(\frac{1}{n}\sum _{j=0}^n(f_*)^j(\mu _0)) -  \frac{1}{n}\sum _{j=0}^n(f_*)^j(\mu _0)\\
&\geq& \frac{1}{n}\sum _{j=0}^n(f_*)^{j+1}(\mu _0)-\frac{1}{n}\sum _{j=0}^n(f_*)^j(\mu _0)\\
&=&\frac{1}{n}(f_*)^{n+1}(\mu _0)-\frac{1}{n}\mu _0,
\end{eqnarray*}
and the latter converges to $0$ in $SM(X)$. Therefore, if $\mu =\lim _{j\rightarrow\infty}\mu _{n_j}$, then any cluster point of $f_*(\mu _{n,j})$ is $\geq \mu$.  Hence, by part 3 of Theorem \ref{TheoremSubmeasurePushforwardMeromorphic} we have $f_*(\mu)\geq \mu$. 



2) Since $\mu _n=(f_*)^n(\mu _0)$ is a decreasing sequence, it has a limit which we denote by $Inv(\geq \mu _0)$, which is an element of $SM^+(X)$. Moreover, the sequence $f_*(\mu _n)=\mu _{n+1}$  also converges to $Inv(\geq \mu _0)$. We then have that $f_*(Inv(\geq \mu _0))\geq Inv(\geq \mu _0)$ by part 3 of Theorem \ref{TheoremSubmeasurePushforwardMeromorphic}. On the other hand, since $Inv (\geq \mu _0)\leq \mu _n$ for all $n$, it follows that $f_*Inv (\geq \mu _0)\leq f_*(\mu _n)$ for all $n$, and hence $f_*Inv (\geq \mu _0)\leq \lim _n f_*(\mu _n)=Inv (\geq \mu _0)$. Combining all inequalities we obtain $f_*(Inv (\geq \mu _0))=Inv (\geq \mu _0)$. 

To finish the proof, we will show that if $\mu \leq \mu _0$ and $f_*(\mu )=\mu$, then $\mu \leq Inv (\geq \mu _0)$. In fact, under the assumptions about $\mu$ we have
\begin{eqnarray*}
\mu =(f_*)^n(\mu) \leq (f_*)^n(\mu _0)
\end{eqnarray*}
for all positive integers $n$. Hence, by taking the limit we obtain that $\mu \leq Inv (\geq \mu _0)$.

 3) We can assume that the mass of $\mu _0$ is $1$. The set $\mathcal{G}:=\{\mu \in SM^+(X):~ \mu \geq \mu _0,~f_*(\mu )= \mu \}$ is non-empty because $\sup _{x\in X}\delta _x$ is one of its elements.

Now let $\mathcal{H}=\{\nu\in M^+(X):~\nu \leq \mu, ~\forall \mu \in \mathcal{G}\}$. Note that $\mathcal{H}$  is non-empty because if $\nu \leq \mu _0$, then $\nu \in \mathcal{H}$. We define $Inv(\geq \mu _0)=\sup _{\nu \in \mathcal{H}}\nu$. Then $Inv(\geq \mu _0)$ is in $SM^+(X)$, and it is the largest element of $SM^+(X)$ which is smaller than $\mu$ for all $\mu \in \mathcal{G}$. Moreover, by construction we see easily that $Inv(\geq \mu _0)\geq \mu _0$. 

We now finish the proof by showing that $f_*(Inv(\geq \mu _0))=Inv(\geq \mu _0)$. First, we show that $Inv(\geq \mu _0)\geq f_*(Inv(\geq \mu _0))$. In fact, it is easy to check that if $\nu \in M^+(X)$ is so that $\nu \leq Inv(\geq \mu _0)$, then $\nu \in \mathcal{H}$. Hence, by part 5 of Theorem \ref{TheoremSubmeasurePushforwardMeromorphic} we have that for all $\mu \in \mathcal{G}$
\begin{eqnarray*}
f_*(Inv(\geq \mu _0))=\sup _{\nu \in \mathcal{H}}f_*(\nu )\leq f_*(\mu)=\mu . 
\end{eqnarray*}
Therefore, by the definition of $Inv(\geq \mu _0)$, we get that $f_*(Inv(\geq \mu _0))\leq Inv(\geq \mu _0)$. Therefore, by part 2 above, we have that if $\mu _{\infty}$ is the limit point of $\mu _n=(f_*)^n(Inv(\geq \mu _0))$, then $f_*(\mu _{\infty})=\mu _{\infty}$. Moreover $Inv(\geq \mu _0)\geq \mu _{\infty}$.

On the other hand, from the fact that $Inv(\geq \mu _0)\geq \mu _0$ and $f_*(\mu _0)\geq \mu _0$ we have
\begin{eqnarray*}
\mu _n =(f_*)^n(Inv(\geq \mu _0))\geq (f_*)^n(\mu _0)\geq \mu _0,
\end{eqnarray*}
for all $n$. Taking the limit we obtain $\mu _{\infty}\geq \mu _0$, that is $\mu _{\infty}\in \mathcal{G}$. Hence, by definition $\mu _{\infty}\geq Inv(\geq \mu _0)$. 

From the above two inequalities, we deduce that $ Inv(\geq \mu _0) =\mu _{\infty}$, which implies that $ Inv(\geq \mu _0)$ is the smallest element of $\mathcal{G}$.

\end{proof}

Now we consider the case where $f$ is good with respect to iterates, see Definition \ref{DefinitionGoodIterateMaps}. We use the same notations $\Gamma _{f,\infty}$, $\sigma _f:~\Gamma _{f,\infty}\rightarrow \Gamma _{f,\infty}$ and $\pi _1:\Gamma _{f,\infty}\rightarrow X$ as in the paragraph after Definition \ref{DefinitionGoodIterateMaps}. The next results relate invariant submeasures of $f$ and those of $\sigma _f$. 

\begin{proposition}
If $\hat{\mu}$ is a positive strong submeasure on $\Gamma _{f,\infty}$, then 
\begin{eqnarray*}
f_* (\pi _1)_*(\hat{\mu})\geq (\pi _1)_*(\phi _f)_*(\hat{\mu}).
\end{eqnarray*}
In general, we have that $f_* (\pi _1)_*(\hat{\mu})\not= (\pi _1)_*(\phi _f)_*(\hat{\mu})$, even if $\hat{\mu}$ is a measure. 
\label{PropositionKeyEntropy}\end{proposition}
\begin{proof}[Proof of Proposition \ref{PropositionKeyEntropy}] By part ii of Theorem \ref{TheoremSubmeasurePushforwardMeromorphicShortVersion} and the fact that $\phi _f$ and $\pi _1$ are continuous maps, we need to consider only the case where $\hat{\mu}$ is a positive measure. We need to show that for all $\varphi \in C^0(X)$ 
\begin{eqnarray*}
f_* (\pi _1)_*(\hat{\mu })(\varphi )\geq (\pi _1)_*(\phi _f)_*(\hat{\mu} )(\varphi ). 
\end{eqnarray*}
By definition 
\begin{eqnarray*}
f_* (\pi _1)_*(\hat{\mu} )(\varphi )=\inf _{\psi \in C^0(X,\geq f^*(\varphi ))}(\pi _1)_*(\hat{\mu} )(\psi ) = \inf _{\psi \in C^0(X,\geq f^*(\varphi ))} \hat{\mu} (\pi _1^*(\psi )). 
\end{eqnarray*}
Let $U=X\backslash I(f)$. Then $U$ is an open dense subset of $X$ and $\pi _1^{-1}(U)$ is an open dense subset of $\Gamma _{f,\infty}$. Note that $\pi _1^*\circ f^*(\varphi )=\phi _f^*\pi _1^*(\varphi )$ on $\pi _1^{-1}(U)$. Moreover, note that the function $\phi _f^*\pi _1^*(\varphi )$ is continuous on the whole $\Gamma _{f,\infty}$. Therefore, for every $\psi \in C^0(X,\geq f^*(\varphi ))$, we have that $\pi _1^*(\psi )\geq \phi _f^*\pi _1^*(\varphi )$. This implies that
\begin{eqnarray*}
f_* (\pi _1)_*(\hat{\mu} )(\varphi )\geq  \hat{\mu} (\phi _f^*\pi _1^*(\varphi ))=(\pi _1)_*(\phi _f)_*(\hat{\mu} )(\varphi ). 
\end{eqnarray*}

If we choose an example $f:X\dashrightarrow X$ having a point $x\in I(f)$ such that $f_*(\delta _x)$ is not a measure, then for any $\hat{x}\in \Gamma _{f,\infty}$  so that $\pi _1(\hat{x})=x$ and $\hat{\mu}=\delta _{\hat{x}}$, we have $f_* (\pi _1)_*(\hat{\mu })\not= (\pi _1)_*(\phi _f)_*(\hat{\mu} )$, since the LHS is not a measure while the RHS is a measure. 
\end{proof}

\begin{theorem} Let $f:X\dashrightarrow X$ be a continuous open-dense defined selfmap of a compact metric space, which is good with respect to iterates. Let $0\not= \mu _0\in SM^+(X)$.

1) Let $\hat{\mu}$ be a $\phi _f$-invariant positive strong submeasure. Then $\mu =(\pi _1)_*(\hat{\mu})$ satisfies: $f_*(\mu )\geq \mu$.
 

2) If $f_*(\mu _0)=\mu _0$, then there exists a non-zero measure $\hat{\mu _0}$ on $\Gamma _{f,\infty}$ so that $(\phi _f)_*(\hat{\mu _0})=\hat{\mu _0}$, $||\hat{\mu _0}||=||\mu _0||$ and $(\pi _1)_*(\hat{\mu _0})\leq \mu _0$. Moreover, the set $\{\hat{\mu} \in SM^+(\Gamma _{f,\infty}):~(\pi _1)_*(\hat{\mu})\leq \mu ,~(\phi _f)_*(\hat{\mu} )= \hat{\mu}\}$ has a largest element, denoted by $Inv(\pi _1, \mu)$. 


\label{TheoremInvariantMeasures}\end{theorem}
\begin{proof}
1) This follows immediately from Proposition \ref{PropositionKeyEntropy}. 


2) Choose any non-zero measure $\nu _0$ on $\Gamma _{f,\infty }$ so that $(\pi _1)_*(\nu _0)\leq \mu$. Then  by Proposition  \ref{PropositionKeyEntropy} we have $$(\pi _1)_*(\phi _f)_*(\nu _0)\leq f_*(\pi _1)_*(\nu _0)\leq f_*(\mu) =\mu.$$
Hence any cluster point $\nu$ of the Cesaro's averages $\frac{1}{n}\sum _{j=0}^n(\phi _f)_*(\nu _0)$ will satisfy $(\pi _1)_*(\nu )\leq \mu$. Now to construct on ethat has the same mass as $\mu$, we choose a sequence of positive measures $\mu _n$ with the property $\mu _n\leq \mu$ and $||\mu _n||\rightarrow ||\mu ||$. For each $\mu _n$, construct a corresponding $\nu _n$ on $\Gamma _{f,\infty}$ as above. Then any cluster point of $\nu$ will satisfy what desired. 

Since $\phi _f$ is continuous, it follows that $(\phi _f)_*(\nu )=\nu$. If we define $Inv(\pi _1,\mu)=\sup \{\hat{\mu} \in SM^+(\Gamma _{f,\infty}):~(\pi _1)_*(\hat{\mu})\leq \mu ,~(\phi _f)_*(\hat{\mu} )= \hat{\mu}\}$, then it is also an element of $SM^+(\Gamma _{f,\infty})$. Moreover, since $\phi _f$ is continuous, we can check easily that $Inv(\pi _1,\mu)$  is $\phi _f$-invariant. 

\end{proof}

In summary, we have several canonical ways to associate invariant positive strong submeasures on $X$ or $\Gamma _{f,\infty}$. If $\hat{\mu}$ is an invariant positive strong submeasure on $\Gamma _{f,\infty}$, then we obtain an invariant positive strong submeasure $Inv(\geq (\pi _1)_*(\hat{\mu}))$ on $X$. Conversely, if $\mu$ is an invariant positive strong submeasure on $X$, then we obtain an invariant positive strong submeasure $Inv(\pi _1, \mu )$ on $\Gamma _{f,\infty}$. If $\mu$ is any positive strong submeasure on $X$ and $\mu _{\infty}$ is any cluster point of Cesaro's average $\frac{1}{n}\sum _{j=0}^n(f_*)^j(\mu _{\infty})$, then we obtain an invariant positive strong submeasure $Inv (\geq \mu _{\infty})$.

\section{Topological and measure theoretic entropy}

In this section we define topological and measure theoretic entropy for continuous open-dense defined selfmaps of compact metric spaces. We start with the case of continuous maps of compact Hausdorff spaces to provide some intuitions and ideas. 

\subsection{The case of continuous maps on compact Hausdorff spaces}
 
 We first recall relevant definitions about entropy of an invariant measure and the Variational Principle, see \cite{goodwyn, goodman}. Let $X$ be a compact Hausdorff space and $f:X\rightarrow X$ a continuous map. Let $\mu$ be a probability Borel-measure on $X$ which is invariant by $f$, that is $f_*(\mu )=\mu$. We next define the entropy $h_{\mu}(f)$ of $f$ with respect to $\mu$. We say that a finite collection of Borel sets $\alpha $ is a $\mu$-partition if $\mu (\bigcup _{A\in \alpha }A)=1$ and $\mu (A\cap B)=0$ whenever $A,B\in \alpha$ and $A\not= B$. Given a $\mu$-partition $\alpha$, we define
 \begin{eqnarray*}
 H_{\mu}(\alpha )=-\sum _{A\in \alpha}\mu (A)\log \mu (A). 
 \end{eqnarray*}
 If $\alpha$ and $\beta $ are $\mu$-partitions, then $\alpha V \beta :=\{A\cap B:$ $A\in \alpha$ and $B\in \beta \}$ is also a $\mu$-partition. Similarly, for all $n$ the collection $V_{i=0}^{n-1}f^{-(i)}(\alpha )$ $:=\{A_0\cap A_1\cap \ldots \cap A_{n-1}:$ $A_0\in \alpha ,$ $A_1\in f^{-1}(\alpha ),$ $\ldots ,$ $A_{n-1}\in f^{-(n-1)}(\alpha )\}$ is also a $\mu $-partition. We define: 
 \begin{equation}
 h_{\mu }(f,\alpha )=\lim _{n\rightarrow\infty}\frac{1}{n}H_{\mu}(V_{i=0}^{n-1}f^{-(i)}(\alpha ))=\inf _n\frac{1}{n}H_{\mu}(V_{i=0}^{n-1}f^{-(i)}(\alpha )),
 \label{EquationMeasureEntropy}\end{equation} 
(the above limit always exists) and 

$h_{\mu}(f):=$ $\sup \{h_{\mu (\alpha )}:$ $\alpha$ runs all over $\mu$-partitions$\}$. 

Recall (see \cite{rudin}) that a Borel-measure $\mu$ of finite mass is regular if for every Borel set $E$: i) $\mu (E)=\inf \{\mu (V):$ $V$ open, $E\subset V\}$, and ii) $\mu (E)=\sup \{\mu (K):$ $K$ compact, $K\subset E\}$. Note that if $X$ is a compact metric space, then any Borel measure of finite mass is regular.   

The Variational Principle, an important result on dynamics of continuous maps, is as follows \cite{goodwyn, goodman}.
\begin{theorem}
Let $X$ be a compact Hausdorff space and $f:X\rightarrow X$ a continuous map. Then 

$h_{top}(f)=$ $\sup \{h_{\mu }(f):$ $\mu$ runs all over probability regular Borel measures $\mu$ invariant by $f\}$. 
\label{TheoremVariationalPrincipleMeasure}\end{theorem}
It is known, however, that the supremum in the theorem may not be attained by any such invariant measure $\mu$, even if $X$ is a compact metric space. In contrast, here, we show that with an appropriate definition of entropy, positive strong submeasures also fit naturally with the Variational Principle and provide the desired maximum. 

First, we note that even though we defined in previous sections strong submeasures only for compact metric spaces, this definition extends easily to the case of compact Hausdorff spaces. We still have the Hahn-Banach theorem on compact Hausdorff spaces, and hence we can define a strong submeasure $\mu$ by one of the following two equivalent definitions: i) $\mu$ is a bounded and sublinear operator on $C^0(X)$, or ii) $\mu =\sup _{\nu \in \mathcal{G}}\nu$, where $\mathcal{G}$ is a non-empty collection of signed regular Borel-measures on $X$ whose norms are uniformly bounded from above. Such a strong submeasure $\mu$ is positive if moreover $\mathcal{G}$ in ii) can be chosen to consist of only positive measures. 

As in the previous sections, we can then define the pushforward of $\mu$ by a continuous map $f:X\rightarrow X$. We have the following property, which is stronger than 5) of Theorem \ref{TheoremSubmeasurePushforwardMeromorphic}. The proof of the result is similar to, and simpler than that of Theorem \ref{TheoremSubmeasurePushforwardMeromorphic}, and hence is omitted. 
\begin{lemma}
Let $X$ be a compact Hausdorff space and $f:X\rightarrow X$ a continuous map. Let $\mu$ be a positive strong submeasure on $X$, and assume that $\mathcal{G}$ is any non-empty collection of positive measures on $X$ such that  $\mu =\sup _{\nu \in \mathcal{G}}\nu$. Then
\begin{eqnarray*}
f_*(\mu )=\sup _{\nu \in \mathcal{G}}f_*(\nu ). 
\end{eqnarray*}
\label{LemmaPushforwardStrongSubmeasure}\end{lemma}
 
Now we define an appropriate notion of entropy for a positive strong submeasure $\mu$ which is invariant by $f$. A first try would be to naively  adapt (\ref{EquationMeasureEntropy}) to the more general case of positive strong submeasures, and then proceed as before. However, this is not appropriate, as the readers can readily check with the simplest case of identity maps on spaces with infinitely many points. In this case, there are many positive strong submeasures with mass $1$ and invariant by $f$, whose entropy, according to the above definition, can be as large as desired and even can be infinity. On the other hand, recall that the topological entropy of the identity map is $0$. (A refinement of this naive version turns out to be appropriate, see the remarks at the end of this section.)  

We instead proceed as follows. Given $\mu$ a positive strong submeasure invariant by $f$ and any regular measure $\nu \leq \mu$, there is a regular measure $\nu '$ so that $\nu '$ is invariant by $f$, $\nu '\leq \mu$ and $\nu '$ has the same mass as $\nu$. Such a measure $\nu '$ can be constructed as any cluster point of the sequence (Cesaro's average) 
\begin{eqnarray*}
\frac{1}{n}\sum _{i=0}^{n-1}f^i_*(\nu ). 
\end{eqnarray*}
We define $\mathcal{G}(f,\mu )=\{\nu :$ $\nu $ is a regular probability Borel-measure invariant by $f$, and $\nu\leq \mu\}$. Finally, we define the desired entropy as follows:
\begin{equation}
h_{\mu}(f):=\sup _{\nu \in \mathcal{G}(f,\mu )}h_{\nu}(f).
\label{EquationEntropySubmeasure}\end{equation}

We now can prove the Variational Principle for positive strong submeasures. We first show that if $\mu$ is any positive strong submeasure of mass $1$ and invariant by $f$, then $h_{\mu }(f)\leq h_{top}(f)$. To this end, we need only to observe that for any $\nu \in \mathcal{G}(f,\mu )$, then the mass of $\nu $ is $\leq 1$, and hence $h_{\nu }(f)\leq h_{top}(f)$ by Theorem \ref{TheoremVariationalPrincipleMeasure}. 

We finish the proof by showing that there is a positive strong submeasure $\mu$ of mass $1$ and invariant by $f$ so that $h_{\mu }(f)=h_{top}(f)$. To this end, we let $\mathcal{G}=\{\nu :$ $\nu$ is a probability regular measure and is invariant by $f\}$. This set is non-empty, by the construction using Cesaro's averages.  We define $\mu =\sup _{\nu \in \mathcal{G}}\nu$. By Lemma \ref{LemmaPushforwardStrongSubmeasure}, we have that 
\begin{eqnarray*}
f_*(\mu )=\sup _{\nu \in \mathcal{G}}f_*(\nu )=\sup _{\nu}\nu =\mu.
\end{eqnarray*}
Hence, $\mu$ is invariant by $f$. Moreover, from the definition of $\mathcal{G}$, it follows that $\mu$ has mass $1$. Finally, by (\ref{EquationEntropySubmeasure}) and Theorem \ref{TheoremVariationalPrincipleMeasure} we have that $h_{\mu }(f)=h_{top}(f)$. 
 
{\bf Remarks.} Besides fitting naturally with the Variational Principle, definition (\ref{EquationEntropySubmeasure}) is also compatible with the philosophy that a property of a positive strong submeasure $\mu$ should be related to the supremum of the same property of measures $\nu \leq \mu$. We have seen some instances of this philosophy above: the definition of $\mu$ itself is such a supremum, and Lemma \ref{LemmaPushforwardStrongSubmeasure}. Since entropy is related to invariant measures of $f$, in definition (\ref{EquationEntropySubmeasure}) we have restricted to only $\mathcal{G}(f,\mu )$.

While we always have that for a positive strong submeasure $\mu$ invariant by a continuous map $f$ then $\mu \geq \sup _{\nu \in \mathcal{G}(f,\mu)}\nu$, it is not always true that $\mu =\sup _{\nu \in \mathcal{G}(f,\mu )}\nu$. In fact, assume that $f$ is not the identity map and $X$ has at least $2$ elements. Let $\mu =\sup _{x\in X}\delta _x$. Then it is easy to check that $\mu$ has mass $1$ and is $f$-invariant. But it is not true that $\mu =\sup _{\nu \in \mathcal{G}(f,\mu )}\nu$. In fact, assume otherwise that $\mu =\sup _{\nu \in \mathcal{G}(f,\mu )}\nu$. Let $x_0\in X$ be such that $f(x_0)\not= x_0$.  Then there is a sequence $\nu _n\in \mathcal{G}(f,\mu )$ so that $\nu _n(x_0)\rightarrow 1$. Hence, since $\nu _n$ all has mass $\leq 1$, it follows that $\nu _n$ converges to $\delta _{x_0}$. Since $f_*(\nu _n)=\nu _n$ for all $n$, it follows also that $f_*(\delta _{x_0})=\delta _0$. This, in turn, means that $f(x_0)=x_0$, which is  a contradiction.  

While the naive version discussed in the paragraph after Lemma \ref{LemmaPushforwardStrongSubmeasure} of the entropy $h_{\mu}(f,\alpha )$ is not a good one, here we present a refinement which is compatible to  (\ref{EquationEntropySubmeasure}) and hence the Variational Principle. Define $\mathcal{P}$ the set of all finite collections $\alpha$ of Borel sets of $X$. For $\nu \in \mathcal{G}(f,\mu )$, if $\alpha$ is a $\nu$-partition, then we define $\widetilde{h}_{\nu }(f,\alpha )=h_{\nu}(f,\alpha )$ as in (\ref{EquationMeasureEntropy}), otherwise we define $\widetilde{h}_{\nu }(f,\alpha )=0$. By (\ref{EquationEntropySubmeasure}), we have
\begin{eqnarray*}
h_{\mu}(f)&=&\sup _{\nu \in \mathcal{G}(f,\mu )}h_{\nu }(f)\\
&=&\sup _{\nu \in \mathcal{G}(f,\mu )}\sup _{\alpha \in \mathcal{P}}\widetilde{h}_{\nu }(f,\alpha )\\
&=&\sup _{\alpha \in \mathcal{P}}[\sup _{\nu \in \mathcal{G}(f,\mu )}\widetilde{h}_{\nu }(f,\alpha )].
\end{eqnarray*}
This suggests us, in parallel with (\ref{EquationMeasureEntropy}), to define for any finite collection of Borel sets $\alpha$, the quantity 
\begin{eqnarray*}
h_{\mu}(f,\alpha )=\sup _{\nu \in \mathcal{G}(f,\mu )}\widetilde{h}_{\nu }(f,\alpha ).
\end{eqnarray*}
Then, we have, as in the classical case: $h_{\mu}(f)=\sup _{\alpha \in \mathcal{P}}h_{\mu}(f,\alpha )$. 

\subsection{The case of continuous open-dense defined selfmaps which are good with respect to iterates} We now define topological entropy and measure theoretic entropy for a continuous open-dense defined selfmap $f:X\dashrightarrow X$, where $X$ is a compact metric space. The topological entropy of $f$ is given by:
\begin{eqnarray*}
h_{top}(f):= h_{top}(\phi _f).
\end{eqnarray*} 

If $\mu$ is a probability measure on $X$ having no mass on $I(f)$ and is invariant by $f$, then we can define exactly as in the continuous dynamics case a notion of measure entropy $h_{\mu}(f)$, and it is again true that $h_{\mu}(f)\leq h_{top}(f)$. However, if $\mu$ has mass on $I(f)$, then in general it is not known how to define measure entropy $h_{\mu}(f)$, since in the definition of measure entropy we need to use the fact that the preimages by $f^j$ of any Borel $\mu$-partition of $X$ are again Borel $\mu$-partitions of $X$ for all $j=0,1,2,\ldots$. The latter is no longer guaranteed if $I(f)\not=\emptyset$  and $\mu$ has mass on $I(f)$. Also, theoretically the Variational Principle may not hold if we restrict to only measures having no mass on $I(f)$. Likewise, other properties of continuous dynamics, which we mention above, do not hold in the meromorphic setting. For example, there is no guarantee that the Cesaro's average approach can provide us with invariant measures.  See Section 5 for a discussion in the related case of transcendental maps of $\mathbb{C}$. 

However, still some analogs of the above classical results in continuous dynamics hold, if we allow positive strong submeasures in the consideration. First, apply the above results, we now can give the proof of Proposition \ref{PropositionMotivationMeasureEntropy}. 
\begin{proof}[Proof of Proposition \ref{PropositionMotivationMeasureEntropy}] Let $\hat{\mu}\in SM^+(\Gamma _{f,\infty})$ such that $(\pi _1)_*(\hat{\mu})\leq {\mu}$ and $||\mu ||=||\hat{\mu}||$. Then it follows that $(\pi _1)_*(\hat{\mu})= {\mu}$.  Since $\mu $ has no mass on $I_{\infty}(f)$, it follows that $\hat{\mu}$ has no mass on $\pi _1^{-1}(I_{\infty}(f))$. Therefore, the mass of $\hat{\mu}$ is concentrated on the set $\{(x,f(x),f^2(x),\ldots )$ $:~x\in X\backslash I_{\infty}(f)\}$. Since $\pi _1$ is a homeomorphism on $X\backslash I_{\infty}(f)$, it follows that $\hat{\mu}$ is $\phi _f$-invariant and 
\begin{eqnarray*}
h_{\mu}(f)=h_{\hat{\mu}}(\phi _f). 
\end{eqnarray*} 
\end{proof}

Using Theorem \ref{TheoremInvariantMeasures} and the previous subsection, we give the following definition. 
\begin{definition}
Let $\mu$ be a positive strong submeasure of mass $1$ invariant by $f$. We define 
\begin{eqnarray*}
h_{\mu}(f):=\sup _{\nu \in M^+(\Gamma _{f,\infty}), ~(\phi _f)_*(\nu )=\nu ,~ ||\nu ||=1 ,~(\pi _1)_*(\nu )\leq \mu}h_{\nu }(\phi _f).
\end{eqnarray*}
\label{DefinitionMeasureEntropyMeromorphicMap}\end{definition}

By Theorem \ref{TheoremInvariantMeasures} and the previous subsection, this measure entropy has good properties as wanted (such as: monotone in $\mu$, bounded from above by dynamical degrees of $f$, and satisfies the Variational Principle). In particular, the following invariant positive strong submeasure captures the dynamics of $f$: We let $\hat{\mu}_{\phi _f,inv}$ to be the supremum of all $\phi _f$-invariant probability measures. Then $\hat{\mu}_{\phi _f,inv}$ is an $\phi _f$-invariant positive strong submeasure and $h_{top}(f)=h_{top}(\phi _f)=h_{\hat{\mu}_{\phi _f,inv}}(\phi _f)$. Now $\mu _{f,inv}=Inv(\geq (\pi _1)_*(\hat{\mu}_{\phi _f,inv}))$  is an $f$-invariant positive strong submeasure of mass $1$ and $h_{\mu _{f,inv}}(f)=h_{top}(f)$. By the remarks at the end of the previous subsection, we have that $\mu _{f,inv}<\sup _{x\in X}\delta _x$ in general, for example when $f$ is a holomorphic map which is not the identity map.

\section{Applications} In this section we provide some applications of the previous sections. Firstly, we will consider a general dominant meromorphic selfmap. Then, we consider in more detail dominant meromorphic selfmaps in dimension $2$. Finally, we consider the case of transcendental maps of $\mathbb{C}$ and $\mathbb{C}^2$.  

\subsection{Dominant meromorphic maps} Let $X$ be a compact complex variety and $f:X\dashrightarrow X$ a dominant meromorphic map. Then from the topological viewpoint, $f$ is an open-dense defined selfmap which is good with respect to iterates. Therefore, we can define pushforward of a positive strong submeasure by $f$, as well as the notions of topological and measure theoretic entropy. By the Variational Principle for the continuous map $\phi _f$ and the definition of $h_{\mu}(f)$, we have that $h_{\mu}(f)\leq h_{top}(f)$, and the latter is bounded from above by dynamical degrees of $f$ by \cite{dinh-sibony3}. Also, by the properties proven above, it is not hard to see that $\max _{\mu}h_{\mu}(f)=h_{top}(f)$, where $\mu$ runs all over $f$-invariant positive strong submeasures of mass $1$. 

We can also define pullback of positive strong submeasures by $f$, following the description in Section 2.3. We end this section describing in detail the pushforward by meromorphic maps on positive strong submeasures. The next result about a good choice of $\psi \in C^0(X,\geq \varphi )$ for some special bounded upper-semicontinuous functions will be needed for that purpose.
\begin{lemma} Let $X$ be a compact metric space, $A\subset X$ a closed set and $U=X\backslash A$. Let $\varphi$ be a bounded upper-semicontinuous function on $X$ so that $\chi =\varphi |_U$ is continuous on $U$ and $\gamma =\varphi |_A$ is continuous on $A$. For any $U'\subset\subset U$ an open set and $\epsilon >0$,  there is a function $\psi \in C^0(X,\geq \varphi )$ so that: 

i) $\psi |_{U'}=\chi $; ii) $\sup _A|\psi |_A-\gamma |\leq \epsilon$; and iii) $\sup _X|\psi |\leq \sup _X|\varphi |+\epsilon$. 
\label{LemmaGoodContinuousFunctionChoice}\end{lemma}
\begin{proof} Let $\epsilon _1>0$ be a small number to be determined later. Since $\varphi $ is upper-semicontinuous, for each $x\in A$, there is $r_x>0$, which we choose so small that $\overline{U'}\cap \overline{B(x,r_x)}=\emptyset$, so that 
\begin{eqnarray*}
\sup _{y\in U\cap \overline{B(x,r_x)}}\chi (y)\leq \gamma (x)+\epsilon _1.
\end{eqnarray*} 
Since $\gamma$ is continuous on $A$, by shrinking $r_x$ if necessary, we can assume that 
\begin{eqnarray*}
\sup _{x'\in A\cap \overline{B(x,r_x)}}|\gamma (x')-\gamma (x)|\leq \epsilon _1. 
\end{eqnarray*} 
Hence we obtain 
\begin{eqnarray*}
\sup _{y\in U\cap \overline{B(x,r_x)}}\chi (y)\leq \inf _{x'\in A\cap \overline{B(x,r_x)}}\gamma (x')+2\epsilon _1. 
\end{eqnarray*}
The function $\gamma |_{A\cap \overline{B(x,r_x)}}$ can be extended to a continuous function $\gamma _x$ on $\overline{B(x,r_x)}$. We can assume, by shrinking $r_x$ for example, that 
\begin{eqnarray*}
\sup _{x',x"\in \overline{B(x,r_x)}}|\gamma _x(x')-\gamma _x(x")|\leq \epsilon _1. 
\end{eqnarray*} 
Now, since $A$ is compact, we can find a finite number of such balls, say $B(x_1,r_1),\ldots ,B(x_m,r_m)$, which cover $A$. We choose $U"$ another open subset of $X$ so that $U'\subset\subset  U"\subset\subset U$ and so that $U",B(x_1,r_1),\ldots ,B(x_m,r_m)$ is a finite open covering of $X$. Let $\tau ,\tau _1,\ldots ,\tau _m$ be a partition of unity subordinate to this open covering. Then the function $$\psi (x)=\tau (x)\gamma (x)+\sum _{i=1}^m\tau _i(x)[\gamma _{x_i}(x)+4\epsilon _1],$$
with $4\epsilon _1<\epsilon$, satisfies the conclusion of the lemma.   
\end{proof}

\begin{proof}[Proof of Theorem \ref{TheoremPushforwardBlowup}]
Let $B\in Z$ be the exceptional divisor of the blowup. Then $\pi :B\rightarrow A$ is a smooth holomorphic fibration, whose fibres are isomorphic to $\mathbb{P}^{r-1}$ where $r=$ the codimension of $A$. As in Example 1, it can be computed that for $x\in A$ then $\pi _*(\varphi )(x)=\sup _{y\in \pi ^{-1}(x)}\varphi $. Therefore, from what was said about the map $\pi :B\rightarrow A$, it follows that $\pi _*(\varphi )|_A$ is continuous. Then it is easy to see that the upper-semicontinuous function $\pi _*(\varphi )$ satisfies the conditions of Lemma \ref{LemmaGoodContinuousFunctionChoice}. It is easy to check that $\pi ^*(\mu _1)$ is a positive measure on $Z$.  Hence, by the conclusion of Lemma \ref{LemmaGoodContinuousFunctionChoice}, it is easy to see that 
\begin{equation}
\pi ^*(\mu )(\varphi )=\pi ^*(\mu _1)(\varphi )+\mu _2(\pi _*(\varphi )|_A).
\label{EquationDetail}\end{equation}

Now we provide  an explicit choice of the collection $\mathcal{G}$ associated to $\pi ^*(\mu )$ in part 2) of Theorem \ref{TheoremHahnBanach}. From Equation (\ref{EquationDetail}), it suffices to provide such a $\mathcal{G}$ for  $\mu _2$, since then the corresponding collection for $\mu$ will be $\pi ^*{\mu _1}+\mathcal{G}$. Therefore, in the remaining of the proof, we will assume that $\mu =\mu _2$ has support on $A$. 

Define $\psi (\varphi )=\pi _*(\varphi )|_A$, it then follows that  $\psi (\varphi )\in C^0(A)$, and 
\begin{eqnarray*}
\pi ^*(\mu )(\varphi )=\mu (\psi (\varphi )). 
\end{eqnarray*}
We now present an explicit collection $\mathcal{G}$ of positive measures on $X$ so that 
\begin{eqnarray*}
\mu (\psi (\varphi ))=\sup _{\chi \in \mathcal{G}}\chi (\varphi ). 
\end{eqnarray*}

To this end, let us consider for each finite open cover $\{U_i\}_{i\in I}$ of $A$, a partition of unity $\{\tau _i\}$ subordinate to the finite open cover $\{U_i\}$ of $A$, and local continuous sections $\gamma _i:U_i\rightarrow \pi ^{-1}(U_i)$, the following assignment on $B$: 
\begin{eqnarray*}
\chi (\{U_i\}, \{\tau _i\}, \gamma _i)(\varphi )=\mu (H(\varphi )), 
\end{eqnarray*}
 where $H(\varphi )\in C^0(A)$ is the following function 
 \begin{eqnarray*}
 H(\varphi )(x)=\sum _{i\in I}\tau _i(x)\varphi (\gamma _i(x)). 
 \end{eqnarray*}
 Since $H(\varphi )$ is linear and non-decreasing in $\varphi$, it is easy to see that $\chi $ is indeed a measure. Moreover, since $|H(x)|\leq \max _B|\varphi |$, it follows that $||\chi ||\leq ||\mu ||$.

We let $\mathcal{G}$ be the collection of such positive measures. We now claim that for all $\varphi \in C^0(Z)$
\begin{eqnarray*}
\mu (\psi (\varphi ))=\sup _{\chi \in \mathcal{G}}\chi (\varphi ). 
\end{eqnarray*}

We show first $\mu (\psi (\varphi ))\geq \chi (\varphi )$ for all $\chi \in \mathcal{G}$.  In fact, since $\gamma _i(x)\in \pi ^{-1}(x)$ for all $x\in A$, it follows by definition that 
\begin{eqnarray*}
H(\varphi )(x)\leq \sup _{\pi ^{-1}(x)}\varphi =\psi (\varphi )(x),
\end{eqnarray*}
for all $x\in A$. Hence $\mu (\psi (\varphi ))\geq \chi (\varphi )$. 
 
Now we show the converse. Let $\varphi $ be any continuous function on $Z$. Then for any $\epsilon >0$ we can always find a finite open covering $\{U_i\}_{i\in I}$ of $X$, depending on $\varphi$ and $\epsilon$, so that for all $x\in U_i$ we have
\begin{eqnarray*}
|\varphi (\gamma _i(x))-\sup _{\pi ^{-1}(x)}\varphi |\leq \epsilon. 
\end{eqnarray*} 
 It then follows that correspondingly $|H(x)-\psi (\varphi )(x)|\leq \epsilon$ for all $x\in A$. Therefore, for this choice of $\chi \in \mathcal{G}$
 \begin{eqnarray*}
 |\mu (\psi (\varphi ))-\mu (H)|\leq \epsilon ,
 \end{eqnarray*}
 and hence letting $\epsilon \rightarrow 0$  concludes the proof
\end{proof}

\subsection{Meromorphic maps in dimension $2$} In this subsection we mention an application to dynamics of meromorphic maps. Let $f:X\dashrightarrow X$ be a dominant meromorphic map of a compact K\"ahler {\bf surface}. The study of dynamics of such maps is very active. It is now recognised that maps which are algebraic stable (those whose pullback on cohomology group is compatible with  iterates, that is $(f^n)^*=(f^*)^n$ on $H^{1,1}(X)$ for all $n\geq 0$) have good dynamical properties. An important indication of the complexity of such  maps is dynamical degrees defined as follows. Let $\lambda _1(f)$ be the spectral radius of the linear map $f^*:H^{1,1}(X)\rightarrow H^{1,1}(X)$ and let $\lambda _2(f)$ be the spectral radius of the linear map $f^*:H^{2,2}(X)\rightarrow H^{2,2}(X)$. There are two large interesting classes of such maps: those with large topological degree ($\lambda _2(f)>\lambda _1(f)$) and those with large first dynamical degree ($\lambda _1(f)>\lambda _2(f)$). The dynamics of the first class is shown in our paper \cite{dinh-nguyen-truong2} to be as nice as expected. For the second class, the most general result so far belongs to \cite{diller-dujardin-guedj1, diller-dujardin-guedj2, diller-dujardin-guedj3}, who showed the existence of canonical Green $(1,1)$ currents $T^+$ and $T^-$ for $f$, and who used potential theory to prove that the dynamics is nice (in particular, the wedge intersection $T^+\wedge T^-$ is well-defined as a positive measure) if the so-called finite energy conditions on the Green currents are satisfied. While these conditions are satisfied for many interesting subclasses, it is known however that in general they are false \cite{buff}. On the other hand, since it is known that $T^+$ has no mass on curves \cite{diller-dujardin-guedj1}, it follows that the least negative intersection $\Lambda (T^+,T^-)$ defined in \cite{truong4} is a positive strong submeasure, see the proof therein. In summary, we obtain the following result.  
\begin{theorem}
Let $f:X\dashrightarrow X$ be a dominant meromorphic map of a compact K\"ahler surface which is algebraic stable and has $\lambda _1(f)>\lambda _2(f)$. Let $T^{+}$ and $T^{-}$ be the canonical Green $(1,1)$ currents of $f$. Then the least negative intersection $\Lambda (T^+,T^-)$ is in $SM^+(X)$. 
\label{TheoremDynamicsDimension2}\end{theorem}
At the moment we do not know whether $f_*(\Lambda (T^+,T^-))=\Lambda (T^+,T^-)$. (In fact, from Buff's examples mentioned above, we do not expect this to be true in general.) However, by the discussion in the previous sections, we can consider cluster points of Cesaro's average $\frac{1}{n}\sum _{j=1}^n(f_*)^j(\Lambda (T^+,T^-))$, and then obtain the associated invariant positive strong submeasures $\mu$. From such a $\mu$, we then can construct associated invariant measures of $\phi _f$. The significance of these invariant positive strong submeasures and measures, together with their entropies, will be pursued in a future work. (See also next subsection where we work out explicitly the cluster points of Cesaro's average, where we work with a transcendental holomorphic map of $\mathbb{C}^1$ and start with a probability measure $\mu $ on $\mathbb{P}^1$. In this case the only dynamically interesting invariant positive strong submeasure we can obtain is $\sup _{x\in \mathbb{P}^1}\delta _x$.) 

{\bf Remark.} We note a parallel between the least negative intersection and the tangent currents in this situation. Under the same assumptions as in Theorem \ref{TheoremDynamicsDimension2}, it was shown in our joint paper \cite{dinh-nguyen-truong} that the h-dimension (defined in \cite{dinh-sibony2}) between $T^+$ and $T^-$ is $0$, the best possible. 

\subsection{Transcendental maps}

While ergodic properties (entropy, invariant measures) of meromorphic maps on $\mathbb{C}$ and $\mathbb{C}^2$ are intensively studied, it seems that these properties are very rarely discussed or explicitly computed for transcendental maps (such as $f(x)=ae^x+b$) in the literature. (Other properties, though, such as Fatou and Julia sets of transcendental holomorphic maps on $\mathbb{C}$, are extensively studied in the literature by many researchers, starting from the work Fatou \cite{fatou}.) In this section, we show that the same ideas used in the previous sections can be applied to this. 

\subsubsection{The case of transcendental maps on $\mathbb{C}$}

Consider the case of a dominant holomorphic map $f:\mathbb{C}\rightarrow \mathbb{C}$. We let $X=\mathbb{P}^1$. While $f$ cannot be extended to a meromorphic map on $X$ if $f$ is transcendental, the set $U=\mathbb{C}$ where $f$ is defined is a dense open set of $X$, and its image $f(U)$ contains an open dense subset of $X$. Therefore, we can use the results in Sections 2--4 to define the pushforward of positive strong submeasures on $X$. (Remark: In contrast, it is in general unknown how to pullback positive strong submeasures on $X$, since $f$ does not have finite fibres.)

 We can then define $\Gamma _{f,\infty}\subset X^{\mathbb{N}}$ and an associated continuous map $\phi _f:\Gamma _{f,\infty}\rightarrow \Gamma _{f,\infty}$ as before, and define the topological entropy by the formula
 \begin{eqnarray*}
 h_{top}(f):=h_{top}(\phi _f).
 \end{eqnarray*} 
In \cite{benini-fornaess-peters}, the authors used a different definition of topological entropy $h_{top}^*(f)$, defined using only compact subsets of $\mathbb{C}$, and showed that $h_{top}^*(f)=\infty$ for every transcendental map. Since $(\mathbb{C},f)$ is embedded as a subsystem of $(\Gamma _{f,\infty},\phi _f)$, it can be checked that $h_{top}^*(f)\leq h_{top}(\phi _f)$ for the map $\phi _f$. Therefore, we also have $ h_{top}(f)=\infty$. 

Now we discuss in more detail the pushforward of a positive measure on $X$ by a transcendental map $f:X\dashrightarrow X$. Let $x_0$ be the point at infinity of $X=\mathbb{P}^1$. Let $\mu$ be a probability measure on $X$. Then we can write $\mu =\mu _0+a \delta _{x_0}$, where $0\leq a\leq 1$. Then $f_*(\mu _0)$ is well-defined as a measure. Since $f$ is transcendental, it follows by Picard's theorem $f(V\cap \mathbb{C})$ is dense in $X$ for all open neighbourhood $V$ of $x_0$. From this we obtain that $f_*(\delta _{x_0})=\mu _X:=\sup _{x\in X}\delta _x$. We obtain the formula: 
\begin{eqnarray*}
f_*(\mu )=f_*(\mu _0)+a \mu _X. 
\end{eqnarray*}  
We note that $f_*(\mu _X)=\mu _X$. 

It follows that if $\mu _{\infty}$ is a cluster point of the Cesaro's averages $\frac{1}{n}\sum _{j=0}^n(f_*)^j(\mu)$, where $\mu$ is a probability measure, then we can write $\mu _{\infty}=\widetilde{\mu _{\infty}}+b\delta _{x_0} +a \mu _X$.  Here $\widetilde{\mu _{\infty}}$ has no mass on $x_0$, $a,b\geq 0$, $a+b\leq 1$. It can be checked easily that $\widetilde{\mu _{\infty}}$ is $f$-invariant. Hence $f_*(\mu _{\infty})=\widetilde{\mu _{\infty}}+(a+b)\mu _X$ is $f$-invariant.  

By Picard's theorem for essential singularities of one complex variable, if $\nu$ is a measure with no mass on $x_0$ and is $f$-invariant, then the support of $\nu$ has at most $2$ points. Therefore $h_{\nu}(f)=0$.  Hence, $f$ has no interesting invariant measure with no mass on $x_0$.  The above calculation then shows that the only interesting invariant positive strong submeasure for $f$ is $\mu _X$. 

\subsubsection{The case of transcendental maps on $\mathbb{C}^2$}

Now, we can apply the above arguments for the case of a transcendental map $f:\mathbb{C}^2\rightarrow \mathbb{C}^2$. The only caveat is that now we have not only one, but many, compact K\"aler surfaces $X$ containing $\mathbb{C}^2$ as an open dense set. We denote by $\mathcal{K}(\mathbb{C}^2)$ the set of all such compact K\"ahler surfaces $X$. For each $X\in \mathcal{K}(\mathbb{C}^2)$, we denote by $f_X$ the associated one on $X$ (note that, as before, the map $f_X$ is not defined on the whole $X$ but only on an open dense subset). Then following the idea in \cite{truong3}, we will define a notion of K\"ahler topological entropy for $f$ as follows: 
\begin{eqnarray*}
h_{top, \mathcal{K}}(f):=\sup _{X\in \mathcal{K}(\mathbb{C}^2)}h_{top}(f_X). 
\end{eqnarray*}
Here $h_{top}(f_X)$ is, as before, defined using the infinity graph $\Gamma _{f_X,\infty}$, which is the closure of $\{(x,f(x),f^(x),\ldots ):~x\in \mathbb{C}^2\}$ in $X^{\mathbb{N}}$, and the shift map $\phi _{f_X}(x_1,x_2,\ldots )=(x_2, x_3,\ldots )$.

We again conjecture that if $f$ is a generic transcendental map then $h_{top, \mathcal{K}(\mathbb{C}^2)}(f)=\infty$. For a generic transcendental map, probably we can show that $f(U)$ is dense in $\mathbb{C}^2$ for all open neighborhood of a point at infinity. In this case, for any $X\in \mathcal{K}(\mathbb{C}^2)$ and  a positive measure $\mu$ on $X$, which is decomposed as $\mu =\mu _0+\mu _1$, where $\mu _0$ has mass only in $\mathbb{C}^2$ and $\mu _1$ has only mass in $X\backslash \mathbb{C}^2$: 
\begin{eqnarray*}
f_*(\mu )=f_*(\mu _0)+a \mu _X,
\end{eqnarray*} 
where $\mu _X=\sup _{x\in X}\delta _x$. Thus we have a similar behaviour of the pushforward of positive measures across all the compactifications $X$. This is in contrast to the case of meromorphic maps in $\mathbb{C}^2$ as we mentioned previously. 

{\bf Remark.} If we generalise the definition of topological entropy of $h_{top}^*$ in \cite{benini-fornaess-peters} to higher dimensions, it follows that for $f(x,y)=(x+1,y^2)$ then $h_{top}^*(f)=0$, which is not the desired value. In fact, consider the compactification $X=\mathbb{P}^2$ of $\mathbb{C}^2$. Then, $(\mathbb{C}^2,f)$ is embedded as a subsystem of $(\Gamma _{f,\infty},\phi _f)$, it follows that $h_{top}^*(f)\leq h_{top}(\phi _f)$. However, for this compactification of $\mathbb{C}^2$, we mentioned before that $h_{top}(\phi _f)=0$.

\section{Conclusion: Why prefer submeasures?} In all our definitions above, we have used submeasures. However, there is also another natural generalisation of measures, that is supermeasures, which are the same as superlinear operators on $C^0(X)$. Which leads to the question: Why using submeasures but not supermeasures? We will now provide a brief explanation of why submeasures are better than supermeasures for the questions considered here. 

First of all, we look at the case of pushforwarding a measure by a meromorphic map. We saw that defining it as a positive strong submeasure is the same as pulling back a continuous function as an upper-semicontinuous function. Proposition \ref{PropositionContinuityPullbackFunctions} and the comments right before it, together with other good properties of this pullback operator, justify this latter point. We also do not know whether the pushforward of supermeasures will have good properties like that for submeasures. 

Second, we look at the case of entropy for maps. In this case, submeasures go along very well with the Variational Principle. As we will see in Section 4, for any continuous map $f:X\rightarrow X$ of a compact metric space, there is a positive strong submeasure $\mu$ of mass $1$ so that $h_{\mu}(f)=h_{top}(f)$. If we use supermeasures instead, we would need to define measure entropy as follows: for a supermeasure $\mu$, we have $\tilde{h}_{\mu}:=\inf _{\nu \in M^+(X),~\nu \geq \mu}h_{\nu }(f)$. But then it is not true in general that there is a supermeasure $\mu$ of mass $1$ so that $\tilde{h}_{\mu}(f)=h_{top}(f)$. (In fact, there are examples of $f:X\rightarrow X$ such that $h_{top}(f)>h_{\nu}(f)$ for all invariant probability measures $\nu$.) Moreover, for a meromorphic map, it is not clear that the similar construction with either Cesaro's average procedure on $X$ or invariant measures on $\Gamma _{f,\infty}$ would lead to dynamically interesting invariant supermeasures.


\begin{thebibliography}{xx} 

\bibitem{adler-konheim-mcandrew}
R. L. Adler, A. G. Konheim and M. H. McAndrew,   \textit{Topological entropy}, Trans. Amer. Math. Soc. 114 (1965), 309--319.





\bibitem{baire}{R. Baire,} \textit{Lesons sur les fonctions discontinues, profess\'ees au coll\`ege de France,} Gauthier-Villars (1905).



\bibitem{benini-fornaess-peters}{A. M. Benini, J. E. Forn\ae ss and H. Peters}  {\it Entropy of transcendental entire functions}, arXiv: 1808.06360.  


\bibitem{buff}{X. Buff,} \textit{Courants dynamiques pluripolaires,} Ann. Toulouse Math. 20 (2011), no 1, 203--214.







\bibitem{diller-dujardin-guedj3}{J. Diller, R. Dujardin and V. Guedj,} \textit{Dynamics of meromorphic maps with small topological degrees III: geometric currents and ergodic theory,} Ann. Sci. Ecole Norm. Sup. 43 (2010), 235--378.

\bibitem{diller-dujardin-guedj2}{J. Diller, R. Dujardin and V. Guedj,} \textit{Dynamics of meromorphic maps with small topological degrees II: energy and invariant measures,} Comm. Math. Helv. 86 (2011), 277--316.

\bibitem{diller-dujardin-guedj1}{J. Diller, R. Dujardin and V. Guedj,} \textit{Dynamics of meromorphic maps with small topological degrees I: from currents to cohomology,} Indianan U. Math. J. 59 (2010), 521--561.

\bibitem{dinh-nguyen-truong}{T.-C. Dinh, V.-A. Nguyen and T. T. Truong,} \textit{Growth of the number of periodic points of meromorphic maps,} Bull. London Math. Soc. to appear. arXiv: 1601.03910.

\bibitem{dinh-nguyen-truong2}{T.-C. Dinh, V.-A. Nguyen and T. T. Truong,} \textit{Equidistribution for meromorphic maps with dominant topological degree,} Indiana Uni. J. Math. 64 (2015), no 6, 1805--1828.

\bibitem{dinh-sibony2}{T.-C. Dinh and N. Sibony,} \textit{Density of positive closed currents, a theory of non-generic intersection,} J. Alg. Geom. to appear. arXiv: 1203.5810.



\bibitem{dinh-sibony3}{T.-C. Dinh and N. Sibony,} \textit{Regularization of currents and entropy,} Ann. Sci. Ecole Norm. Sup. (4), 37 (2004), no 6, 959--971.

\bibitem{doob}{J. L. Doob,} \textit{Measure theory,} Graduate text in mathematics 143, Springer - Verlag, New York 1994.

\bibitem{fatou}{P. Fatou,} \textit{Sur l'iteration des fonctions transcendantes entires,} Acta Math. 47 (1926), vol 4, 337--370.


\bibitem{goodman}{T. N. T. Goodman,} \textit{Relating topological entropy and measure entropy,} Bull. London. Math. Soc. 3 (1971), 176--180.

\bibitem{goodwyn}{L. W. Goodwyn,} \textit{Comparing topological entropy with measure-theoretic entropy,} American Journal of Mathematics, vol 94, no 2 (April 1972), 366--388.







\bibitem{meo} M. M\'eo, \textit{Inverse image of a closed positive current by a surjective analytic map}, (in French), C. Acad. Sci. Paris Ser. I Math. 322 (1996), no 2, 1141--1144. 

\bibitem{rudin2}{W. Rudin,} \textit{Functional analysis,} 2nd edition, Mc Graw - Hill, New York 1991.

\bibitem{rudin}{W. Rudin,} \textit{Real and complex analysis,} 3rd edition, Mc Graw - Hill, New York 1987.


\bibitem{talagrand}{M. Talagrand,} \textit{Maharam's problem,} Ann. of Math. 168 (2008), 981--1009.


\bibitem{truong4}{T. T. Truong,} \textit{Submeasures and several applications,}  arXiv: 1712.02490. 

\bibitem{truong3}{T. T. Truong,} \textit{Etale dynamical systems and topological entropy,}  arXiv: 1607.07412. Accepted in Proc. AMS. 






\end{thebibliography}
\end{document}